\documentclass{amsart}

\usepackage{amssymb}
\usepackage{amsfonts}
\usepackage{amsmath}
\usepackage{amsthm}
\usepackage{graphicx}
\usepackage{mathrsfs}
\usepackage{dsfont}
\usepackage{amscd}
\usepackage{multirow}
\usepackage[all]{xy}
\usepackage[T1]{fontenc}
\usepackage{calligra}
\usepackage{verbatim}
\usepackage[usenames,dvipsnames]{color}
\usepackage[colorlinks=true,linkcolor=Blue,citecolor=Violet]{hyperref}
\usepackage{enumerate}

\newcommand{\N}{{\mathds{N}}}
\newcommand{\Z}{{\mathds{Z}}}

\newcommand{\R}{{\mathds{R}}}
\newcommand{\C}{{\mathds{C}}}
\newcommand{\T}{{\mathds{T}}}

\newcommand{\D}{{\mathfrak{D}}}
\newcommand{\A}{{\mathfrak{A}}}
\newcommand{\B}{{\mathfrak{B}}}

\newcommand{\Lip}{{\mathsf{L}}}
\newcommand{\Lipp}{{\mathsf{R}}}

\newcommand{\propinquity}[1]{{\mathsf{\Lambda}_{#1}}}
\newcommand{\dpropinquity}[1]{{\mathsf{\Lambda}^\ast_{#1}}}

\newcommand{\Kantorovich}[1]{{\mathsf{mk}_{#1}}}

\newcommand{\KantorovichLength}[1]{{\mathrm{mk}\ell_{#1}}}
\newcommand{\KantorovichDist}[4]{{\mathrm{mkD}^{#1}_{#2}\left({#3},{#4}\right)}}

\newcommand{\Haus}[1]{{\mathsf{Haus}_{#1}}}

\newcommand{\StateSpace}{{\mathscr{S}}}

\newcommand{\MongeKant}{{Mon\-ge-Kan\-to\-ro\-vich metric}}

\newcommand{\Lqcms}{{\JLL} quantum compact metric space}
\newcommand{\Qqcms}[1]{{$#1$}--\gQqcms}
\newcommand{\gQqcms}{quasi-Leibniz quantum compact metric space}

\newcommand{\qcms}{quantum compact metric space}

\newcommand{\unit}{1}

\newcommand{\sa}[1]{{\mathfrak{sa}\left({#1}\right)}}

\newcommand{\JLL}{Lei\-bniz}

\newcommand{\dom}[1]{{\operatorname*{dom}\left({#1}\right)}}
\newcommand{\codom}[1]{{\operatorname*{codom}\left({#1}\right)}}
\newcommand{\diam}[2]{{\mathrm{diam}\left({#1},{#2}\right)}}

\newcommand{\norm}[2]{{\left\|{#1}\right\|_{#2}}}
\newcommand{\tunnelset}[3]{{\text{\calligra Tunnels}\,\left[\left({#1}\right)\stackrel{#3}{\longrightarrow}\left({#2}\right)\right]}}

\newcommand{\Jordan}[2]{{{#1}\circ{#2}}}
\newcommand{\Lie}[2]{{\left\{{#1},{#2}\right\}}}

\newcommand{\targetsettunnel}[3]{{\mathfrak{t}_{#1}\left({#2}\middle\vert{#3}\right)}}

\newcommand{\targetsetimage}[3]{{\mathfrak{i}_{#1}\left({#2}\middle\vert{#3}\right)}}
\newcommand{\targetsetforward}[3]{{\mathfrak{f}_{#1}\left({#2}\middle\vert{#3}\right)}}

\newcommand{\worknote}[1]{}
\newcommand{\opnorm}[3]{{\left|\mkern-1.5mu\left|\mkern-1.5mu\left| {#1} \right|\mkern-1.5mu\right|\mkern-1.5mu\right|_{#3}^{#2}}}

\newcommand{\tunnelextent}[1]{{\chi\left({#1}\right)}}

\newcommand{\alg}[1]{{\mathfrak{#1}}}

\theoremstyle{plain}
\newtheorem{theorem}{Theorem}[section]

\newtheorem{step}{Step}
\newtheorem{corollary}[theorem]{Corollary}

\newtheorem{lemma}[theorem]{Lemma}
\newtheorem{proposition}[theorem]{Proposition}

\newtheorem{theorem-definition}[theorem]{Theorem-Definition}

\theoremstyle{definition}
\newtheorem{definition}[theorem]{Definition}

\newtheorem{hypothesis}[theorem]{Hypothesis}

\newtheorem{notation}[theorem]{Notation}
\newtheorem*{notation*}{Notation}

\theoremstyle{remark}

\newtheorem{example}[theorem]{Example}

\newtheorem{remark}[theorem]{Remark}
\newtheorem{summary}[theorem]{Summary}

\renewcommand{\geq}{\geqslant}
\renewcommand{\leq}{\leqslant}

\numberwithin{equation}{section}

\allowdisplaybreaks[4]

\begin{document}

\title[Actions of Categories and the Gromov-Hausdorff Propinquity]{Actions of Categories by Lipschitz morphisms on limits for the Gromov-Hausdorff Propinquity}
\author{Fr\'{e}d\'{e}ric Latr\'{e}moli\`{e}re}
\thanks{This work is part of the project supported by the grant H2020-MSCA-RISE-2015-691246-QUANTUM DYNAMICS}
\email{frederic@math.du.edu}
\urladdr{http://www.math.du.edu/\symbol{126}frederic}
\address{Department of Mathematics \\ University of Denver \\ Denver CO 80208}

\date{\today}

\begin{abstract}
We prove a compactness result for classes of actions of many small categories on quantum compact metric spaces by Lipschitz linear maps, for the topology of the covariant Gromov-Hausdorff propinquity. In particular, our result applies to actions of proper groups by Lipschitz isomorphisms on quantum compact spaces. Our result provides a first example of a structure which passes to the limit of quantum metric spaces for the propinquity, as well as a new method to construct group actions, including from non-locally compact groups seen as inductive limits of compact groups, on unital C*-algebras. We apply our techniques to obtain some properties of closure of certain classes of {\gQqcms s} for the propinquity.
\end{abstract}
\maketitle


The noncommutative Gromov-Hausdorff propinquity \cite{Latremoliere13,Latremoliere13b,Latremoliere14,Latremoliere16c,Latremoliere17a} is a distance which induces a topology on  the hyperspace of quantum compact metric spaces, with sights on applications to mathematical physics.

A natural yet delicate question is to determine which properties do pass to the limit of a sequence of quantum metric spaces. Partial answers to this question can for instance help us determine the closure of a class of quantum metric spaces. Most importantly, the propinquity is intended to be used to construct theories for physics from finite dimensional models \cite{Latremoliere13c,Latremoliere15d,Rieffel15} or by using other form of approximations \cite{Latremoliere15c,Latremoliere16}, possibly using the fact that the propinquity is complete. Thus, we would like to know which properties of the finite approximations would indeed carry to the eventual limit.

The existence of group actions on limits of sequences of quantum compact metric spaces is thus of prime interest. The question can be informally asked: if a sequence of metrized compact groups act on a Cauchy sequence of {\gQqcms s}, and if these groups converge in some appropriate sense, then does the limit carry an action of the limit group? Asked in such a manner, the answer is trivially yes --- one may involve trivial actions --- but it is easy and rather natural to impose a form of nontriviality using the underlying quantum metrics. Indeed, we can control the dilation of the automorphisms \cite{Latremoliere16b} arising from the actions at each level to ensure that any sort of limit would also be non-trivial. We also note that much efforts have been dedicated to constructing full matrix approximations of certain homogeneous spaces \cite{Rieffel01,Rieffel10c,Rieffel11,Rieffel15} with the idea that finite dimensional quantum metric spaces are the noncommutative analogues of finite spaces, though they can carry actions of compact Lie groups such as $SU(n)$, and therefore have highly non-trivial symmetries --- thus, offering an option to work with ``finite spaces'' with continuous symmetries. From this perspective, our present work addresses a sort of dual question: if we do use approximations of a physical model using, say, finite dimensional algebras with some specified symmetries, what are sufficient conditions for these symmetries to  still be present at the limit for the propinquity?

We establish a formal answer to this problem, in a more general context: rather than only working with groups, we prove the result for an appropriate notion of action of small categories --- thus, our result applies to groups actions, groupoid actions on a set of C*-algebras, and semigroups actions, among others. We also note that our result can be applied to prove the existence of non-trivial actions of inductive limits of groups, even if this inductive limit is not locally compact. While left for further research, the results of this paper hint at a possible application of techniques from noncommutative metric geometry to the study of actions of non-locally compact groups on C*-algebras.

As an application of our work, we exhibit closed classes of {\gQqcms s}, which is in general a very difficult problem to solve. The key observation is that with our method, the limit of ergodic actions is ergodic, in the sense described in this paper --- an observation which is quite powerful. For instance, we can deduce from our work and known facts on ergodic actions of $SU(2)$ that the closure of many classes of {\gQqcms s} whose quantum metrics are induced by ergodic actions of $SU(2)$ consists only of type I C*-algebras. In contrast, the closure of classes of fuzzy tori whose quantum metric arises from a fixed length function of the tori, and ergodic actions of closed subgroups of the tori, consists only of fuzzy and quantum tori (and from our own prior work, include all of the quantum tori). More generally, our work opens a new direction for finding unital C*-algebras on which compact groups act ergodically, by investigating the closure for the propinquity of finite dimensional C*-algebras on which a given compact group acts ergodically. The study of this problem from a metric property will be the subject of subsequent papers, and represents an exciting potential development for our project.

We begin this paper with a review of our theory for the hyperspace of quantum metric spaces. We then prove our main result in the context of actions of countable category, appropriately defined. Our result does not per say require an action of a category, but rather an approximate form of it. We then turn to various applications of our main result, concerning actions of groups, groupoids, and semigroups.

\bigskip

The following notations will be used throughout this paper.
\begin{notation*}
If $(E,d)$ is a metric space, the diameter $\sup\{d(x,y) : x,y\in E\}$ of $E$ is denoted by $\diam{E}{d}$.

The norm of a normed vector space $X$ is denoted by $\|\cdot\|_X$ by default. The unit of a unital C*-algebra $\A$ is denoted by $\unit_\A$. The state space of a C*-algebra $\A$ is denoted by $\StateSpace(\A)$ while the real subspace of self-adjoint elements of $\A$ is denoted by $\sa{\A}$.

Given any two self-adjoint elements $a,b \in \sa{\A}$ of a C*-algebra $\A$, the \emph{Jordan product} $\frac{a b + b a}{2}$, i.e. the real part of $a b$, is denoted by $\Jordan{a}{b}$, while the \emph{Lie product} $\frac{a b - b a}{2i}$ --- the imaginary part of $a b$ --- is denoted by $\Lie{a}{b}$. The notation $\circ$ will often be used for composition of various maps, but will always mean the Jordan product when applied to elements of a C*-algebra, as is customary. The triple $(\sa{\A},\Jordan{\cdot}{\cdot},\Lie{\cdot}{\cdot})$ is a Jordan-Lie algebra \cite{landsman98}.

Last, for any continuous linear map $\varphi : E \rightarrow F$ between two normed vector spaces $E$ and $F$, the norm of $\varphi$ is denoted by $\opnorm{\varphi}{E}{F}$ --- or simply $\opnorm{\varphi}{}{E}$ if $E = F$.
\end{notation*}

\section{Background}

Compact quantum metric spaces were introduced by Rieffel \cite{Rieffel98a, Rieffel99}, inspired by the work of Connes \cite{Connes89}, as a noncommutative generalization of algebras of Lipschitz functions over compact metric spaces \cite{Weaver99}. We added to this original approach the requirement that a form of the Leibniz inequality holds for the generalized Lipschitz seminorm in a quantum compact metric space; the upper bound of this inequality is given using a particular type of functions:

\begin{definition}
A function $F : [0,\infty]^4 \rightarrow [0,\infty]$ is \emph{admissible} when:
\begin{enumerate}
\item $F$ is weakly increasing if $[0,\infty]^4$ is equipped with the product order,
\item for all $x,y,l_x,l_y \geq 0$ we have $x l_y + y l_x \leq F(x,y,l_x,l_y)$.
\end{enumerate}
\end{definition}
The second condition in the above definition of an admissible function ensures that the usual Leibniz inequality is the ``best'' upper bound one might expect for our results to hold in general.

A {\gQqcms} is thus defined as follows:

\begin{definition}[{\cite{Connes89, Rieffel98a, Rieffel99, Latremoliere15}}]\label{qqcms-def}
A \emph{\Qqcms{F}} $(\A,\Lip)$, where $F$ is an admissible function, is an ordered pair of a unital C*-algebra $\A$ and a seminorm $\Lip$ defined on a dense Jordan-Lie subalgebra $\dom{\Lip}$ of $\sa{\A}$ such that:
\begin{enumerate}
\item $\{a\in\dom{\Lip} : \Lip(a) = 0\} = \R \unit_\A$,
\item the \emph{\MongeKant} $\Kantorovich{\Lip}$ defined between any two states $\varphi, \psi \in \StateSpace(\A)$ of $\A$ by
\begin{equation*}
\Kantorovich{\Lip}(\varphi,\psi) = \sup\left\{ \left| \varphi(a) - \psi(a) \right| : a \in \dom{\Lip}, \Lip(a) \leq 1 \right\}
\end{equation*}
metrizes the weak* topology of $\StateSpace(\A)$,
\item for all $a,b \in \dom{\Lip}$ we have
\begin{equation*}
\max\left\{ \Lip\left(\Jordan{a}{b}\right), \Lip\left(\Lie{a}{b}\right) \right\} \leq F\left(\|a\|_\A,\|b\|_\A,\Lip(a),\Lip(b)\right) \text{,}
\end{equation*}
\item $\Lip$ is lower semi-continuous on $\sa{\A}$ with respect to $\|\cdot\|_\A$ (where, for all $a\in\sa{\A}\setminus\dom{\Lip}$, we set $\Lip(a)=\infty$).
\end{enumerate}
When $(\A,\Lip)$ is a {\gQqcms}, the seminorm $\Lip$ is called an \emph{L-seminorm}.
\end{definition}

\begin{notation}
If $(\A,\Lip)$ is a {\gQqcms}, then $(\StateSpace(\A),\Kantorovich{\Lip})$ is a compact metric space, so it has finite diameter, which we denote by $\diam{\A}{\Lip}$.
\end{notation}

\begin{example}[Classical Metric Spaces]\label{classical-ex}
Let $(X,d)$ be a compact metric space. For all $f : X \rightarrow \R$, we define
\begin{equation*}
\Lip(f) = \sup\left\{ \frac{|f(x) - f(y)|}{d(x,y)} : x,y \in X, x\not= y \right\}\text{.}
\end{equation*}

The pair $(C(X),\Lip)$ of the unital C*-algebra of $\C$-valued continuous functions on $X$ and the seminorm $\Lip$ is a {\Lqcms}.
\end{example}

\begin{example}[Noncommutative Examples]
Many examples of noncommutative {\gQqcms s} are known to exist, including quantum tori \cite{Rieffel98a,Rieffel02}, curved quantum tori \cite{Latremoliere15c}, AF algebras \cite{Latremoliere15d}, noncommutative solenoids \cite{Latremoliere16b}, hyperbolic groups C*-algebras \cite{Ozawa05}, nilpotent group C*-algebras \cite{Rieffel15b}, Podles spheres \cite{Kaad18}, among others. We also note that a theory for locally compact quantum metric spaces has been developed in \cite{Latremoliere05b,Latremoliere12b}.
\end{example}

Quasi-Leibniz quantum compact metric spaces are objects for several natural categories. For this paper, we will use three different notions of morphisms. We begin with the weakest of these notions, which we will use to state our main theorem at a hopefully reasonable level of generality.

\begin{definition}\label{Lipschitz-linear-def}
Let $(\A,\Lip_\A)$ and $(\B,\Lip_\B)$ be two {\gQqcms s}. A continuous linear map $\varphi : \A \rightarrow \B$ is \emph{Lipschitz} when
\begin{equation*}
\forall a \in \dom{\Lip_\A} \quad \varphi(a) \in \dom{\Lip_\B}
\end{equation*}
and $\varphi(\unit_\A) = \unit_\B$.
\end{definition}

\begin{remark}
A linear Lipschitz map, in particular, maps self-adjoint elements to self-adjoint elements, since it is continuous and maps the domain of an L-seminorm to self-adjoint elements --- while domains of L-seminorms are dense in the self-adjoint part of C*-algebras.
\end{remark}

The terminology of Definition (\ref{Lipschitz-linear-def}) is explained by the following corollary of our work in \cite{Latremoliere16}:

\begin{proposition}
If $(\A,\Lip_\A)$ and $(\B,\Lip_\B)$ are {\gQqcms s}, and if $\varphi : \A \rightarrow \B$ is a continuous, self-adjoint linear map with $\varphi(\unit_\A) = \unit_\B$, then the following are equivalent:
\begin{enumerate}
\item $\varphi$ is Lipschitz,
\item there exists $k > 0$ such that $\Lip_\B\circ\varphi \leq k \Lip_\A$.
\end{enumerate}
\end{proposition}

\begin{proof}
This result follows by \cite[Theorem 2.1]{Latremoliere16}.
\end{proof}

It is straightforward to check that the identity of any {\gQqcms s} is a Lipschitz linear map, and that the composition of Lipschitz linear maps is itself Lipschitz, so we have defined a category, as announced.

Among all the Lipschitz linear functions, we find the \emph{Lipschitz morphisms} \cite{Latremoliere16}:
\begin{definition}
Let $(\A,\Lip_\A)$ and $(\B,\Lip_\B)$ be {\gQqcms s}. A \emph{Lipschitz morphism} $\varphi : (\A,\Lip_\A) \rightarrow (\B,\Lip_\B)$ is a unital *-homomorphism which is also Lipschitz linear.
\end{definition}

Using Lipschitz morphisms as arrows does form a category over the class of {\gQqcms s}, which is in fact the \emph{natural}, default category for our theory. To be more explicit, one may argue that going from just linear maps to actual *-homomorphisms is the key reason to introduce the propinquity, versus other versions of the noncommutative Gromov-Hausdorff distance \cite{Rieffel00,kerr02,Rieffel09}: the propinquity was introduced specifically to allow us to work within the category of C*-algebras with their *-homomorphisms. Moreover, the dual maps induced by unital *-homomorphisms restricts to functions between state spaces, and these functions are Lipschitz for the {\MongeKant} if and only if the *-homomorphism is a Lipschitz morphism, thus completing the geometric picture.

The only reason for our digression via the Lipschitz linear maps is to be able to phrase our main theorem at a slightly higher level of generality, as certain examples, such as semigroups of positive linear maps, have appeared in the literature and present potential examples for applications of our present work.

\medskip

There is a natural distance on the space of unital continuous linear maps between {\gQqcms s}, which generalizes the Monge-Kantorovich length of *-automorphisms discussed in \cite{Latremoliere16}. We will use this distance for two purposes. It will quantify how far the composition of two Lipschitz linear maps is to some other Lipschitz linear map --- a tool used in the hypothesis of our main theorem --- and it will also define a non-triviality test for actions of categories, in the conclusion of our main theorem.

\begin{definition}\label{Monge-Kantorovich-Dist-def}
Let $(\A,\Lip_\A)$ be a {\gQqcms} and $\B$ be a unital C*-algebra. For any two linear maps $\varphi, \psi$ from $\A$ to $\B$ which map $\unit_\A$ to $\unit_\B$, we define the \emph{Monge-Kantorovich distance} between $\varphi$, $\psi$ as:
\begin{equation*}
\KantorovichDist{\Lip_\A}{\B}{\varphi}{\psi} = \sup\left\{ \left\|\varphi(a) - \psi(a)\right\|_{\B} : a\in\dom{\Lip_\A}, \Lip_\A(a) \leq 1 \right\}\text{.}
\end{equation*}

In particular, if $(\A,\Lip_\A) = (\B,\Lip_\B)$ then we write $\KantorovichLength{\Lip_\A}(\varphi)$ for $\KantorovichDist{\Lip_\A}{\A}{\mathrm{id}}{\varphi}$ where $\mathrm{id}$ is the identity automorphism of $\A$.
\end{definition}

In \cite[Proposition 5.2]{Latremoliere16}, we proved that for any {\gQqcms} $(\A,\Lip)$, the function $\KantorovichLength{\Lip}$ is a length function on the group of *-automorphisms of $\A$ which metrizes the pointwise convergence topology. This result and its proof extends readily to the setting of Definition (\ref{Monge-Kantorovich-Dist-def}), so we sketch the proof rather briefly:

\begin{proposition}\label{Monge-Kantorovich-Dist-prop}
Let $(\A,\Lip_\A)$ be a {\gQqcms} and $\B$ be a unital C*-algebra. The function $\KantorovichDist{\Lip_\A}{\B}{\cdot}{\cdot}$ defines a distance on the set of all continuous linear functions from $\A$ to $\B$ which map $\unit_\A$ to $\unit_\B$.
\end{proposition}

\begin{proof}
Let $\mu \in \StateSpace(\A)$. By definition of the {\MongeKant}, for all $a\in\dom{\Lip_\A}$
\begin{equation*}
\|a - \mu(a)\unit_\A\|_\A \leq \diam{\A}{\Lip_\A} \Lip_\A(a)\text{.}
\end{equation*}

Thus for any two continuous linear maps $\varphi,\psi :\A\rightarrow\B$ mapping $\unit_\A$ to $\unit_\B$, and for all $a\in\dom{\Lip_\A}$ with $\Lip_\A(a)\leq 1$
\begin{multline*}
\left\| \varphi(a) - \psi(a) \right\|_\B = \left\|\varphi(a-\mu(a)\unit_\A) - \psi(a-\mu(a)\unit_\A)\right\|_\B \\ \leq \diam{\A}{\Lip_\A}\left(\opnorm{\varphi}{\A}{\B} + \opnorm{\psi}{\A}{\B}\right) \text{.}
\end{multline*}

Thus $\KantorovichDist{\Lip_\A}{\B}{\cdot}{\cdot}$ is finite-valued and bounded on balls for the norm $\opnorm{\cdot}{\A}{\B}$. The triangle inequality and the symmetry properties of $\KantorovichDist{\Lip_\A}{\B}{\cdot}{\cdot}$ are obvious.

If $\KantorovichDist{\Lip_\A}{\B}{\varphi}{\psi} = 0$ then $\varphi$ and $\psi$ agree on the total set $\{a\in\sa{\A}:\Lip_\A(a) \leq 1\}$, and thus by linearity and continuity, they agree on $\A$. Thus $\KantorovichDist{\Lip_\A}{\B}{\cdot}{\cdot}$ is a distance as claimed.

The proof that $\KantorovichDist{\Lip_\A}{\B}{\cdot}{\cdot}$ induces the topology of pointwise convergence follows the same argument as \cite[Claim 5.4, Proposition 5.2]{Latremoliere16}, as the reader can easily check (where pointwise convergence topology means the locally convex topology induced by the family of seminorms defined over the space of linear maps  from $\A$ to $\B$ by $\varphi \mapsto \|\varphi(a)\|_\B$ where $a$ ranges over all of $\A$).
\end{proof}

\begin{remark}
Definition (\ref{Monge-Kantorovich-Dist-def}) and Proposition (\ref{Monge-Kantorovich-Dist-prop}) would make sense and hold if the codomain was some normed vector space, as long as we work on the set of continuous linear maps which map the unit to some fixed vector. This level of generality is not however needed here.
\end{remark}

Now, among all the Lipschitz morphisms, the quantum isometries form a subcategory, and are essential tools  in the construction of the propinquity. Moreover, our chosen notion of isomorphism between {\gQqcms s} will be given by the full quantum isometries --- an important observation to understand the geometry of the propinquity. The notion of a quantum isometry is largely due to Rieffel, and essentially expresses that McShane theorem \cite{McShane34} on extensions of real-valued Lipschitz functions holds in our noncommutative context. This justifies, a posteriori, the emphasis on self-adjoint elements in Definition (\ref{qqcms-def}). Our modification of Rieffel's notion consists of the requirement that quantum isometries be *-homomorphisms, and in particular, full quantum isometries be *-isomorphisms, rather than positive unital linear maps.

\begin{definition}
Let $(\A,\Lip_\A)$ and $(\B,\Lip_\B)$ be two {\gQqcms s}. A \emph{quantum isometry} $\pi : (\A,\Lip_\A) \twoheadrightarrow (\B,\Lip_\B)$ is a *-epimorphism $\pi : \A\twoheadrightarrow\B$ such that for all $b\in\sa{\B}$
\begin{equation*}
\Lip_\B(b) = \inf\left\{ \Lip_\A(a) : \pi(a) = b \right\} \text{.}
\end{equation*}

A \emph{full quantum isometry} $\pi : (\A,\Lip_\A) \rightarrow (\B,\Lip_\B)$ is a *-isomorphism $\pi : \A\rightarrow\B$ such that $\Lip_\B\circ\pi = \Lip_\A$.
\end{definition}

A full quantum isometry is a quantum isometry which is a also a bijection and whose inverse map is a quantum isometry. Quantum isometries are in particular $1$-Lipschitz morphisms between {\gQqcms s} \cite{Latremoliere16c} and they are morphisms for a category whose objects are {\gQqcms s} \cite{Rieffel00}.

\medskip

Our work focuses on the construction of geometries on the hyperspace of all {\gQqcms s}, or phrased a little more formally, on defining metrics on the class of all {\gQqcms s}, or some subclass therein. There are many subtleties involved in these constructions; our focus will be on the dual Gromov-Hausdorff propinquity as introduced in \cite{Latremoliere13b,Latremoliere14}.

The dual Gromov-Hausdorff propinquity is a complete distance, up to full quantum isometry, on the class of {\Qqcms{F}s}, which induces the topology of the Gromov-Hausdorff distance \cite{Edwards75,Gromov81} on the subclass of all classical compact metric spaces. Its construction is summarized as follows.

We first generalize the idea of an isometric embedding of two {\gQqcms s} into a third one.

\begin{definition}[{\cite{Latremoliere13b}}]\label{tunnel-def}
Let $(\A,\Lip_\A)$ and $(\B,\Lip_\B)$ be two {\Qqcms{F}s}. An \emph{$F$-tunnel} $\tau = (\D,\Lip_\D,\pi_\A,\pi_\B)$ is given by a {\Qqcms{F}} $(\D,\Lip_\D)$ as well as two quantum isometries $\pi_\A : (\D,\Lip_\D)\twoheadrightarrow (\A,\Lip_\A)$ and $\pi_\B : (\D,\Lip_\D) \twoheadrightarrow (\B,\Lip_\B)$.

The \emph{domain} of $\tau$ is $(\A,\Lip_\A)$ and the \emph{codomain} of $\tau$ is $(\B,\Lip_\B)$.
\end{definition}

We then associate a numerical value to every tunnel.

\begin{notation}
If $\pi : \A\rightarrow\B$ is a positive unital linear map from a unital C*-algebra $\A$ to a unital C*-algebra $\B$, then the map $\varphi \in \StateSpace(\B) \rightarrow \varphi \circ \pi \in \StateSpace(\A)$ is denoted by $\pi^\ast$.
\end{notation}

\begin{notation}
The Hausdorff distance \cite[p. 293]{Hausdorff} induced on the hyperspace of compact subsets of a metric space $(X,d)$ is denoted by $\Haus{d}$; if $X$ is in fact a vector space and $d$ is induced by some norm $\|\cdot\|_X$, then $\Haus{d}$ is denoted by $\Haus{\|\cdot\|_X}$.
\end{notation}

\begin{definition}[{\cite{Latremoliere14}}]\label{extent-def}
Let $(\A_1,\Lip_1)$ and $(\A_2,\Lip_2)$ be two {\gQqcms s}. The \emph{extent} $\tunnelextent{\tau}$ of a tunnel $\tau = (\D,\Lip,\pi_1,\pi_2)$ from $(\A_1,\Lip_1)$ to $(\A_2,\Lip_2)$ is
\begin{equation*}
\max\left\{ \Haus{\Kantorovich{\Lip}}\left(\StateSpace(\D),\pi_j^\ast(\StateSpace(\A_j))\right) : j=1,2\right\} \text{.}
\end{equation*}
\end{definition}

Now, the dual propinquity can in fact be ``specialized'', by which we mean that we may consider only certain tunnels with additional properties. For instance, one might want to work with tunnels involving only strongly Leibniz L-seminorms \cite{Rieffel10c}. The following definition summarizes the properties needed on a class of tunnels to allow for the definition of a version of the dual propinquity. The reader may skip this matter, except for the definition of the inverse of a tunnel in item (2), and simply assume that we will work with all $F$-tunnels for a given choice of an admissible function $F$ in this paper; however, it is worth to point out that our main result applies equally well to any specialization of the propinquity.

\begin{definition}[{\cite{Latremoliere14}}]
Let $F$ be an admissible function. A class $\mathcal{T}$ of tunnels is \emph{appropriate} for a nonempty class $\mathcal{C}$ of {\Qqcms{F}s} when the following properties hold.
\begin{enumerate}
\item If $\tau\in\mathcal{T}$ then $\dom{\tau},\codom{\tau} \in \mathcal{C}$,
\item If $(\A,\Lip_\A)$ and $(\B,\Lip_\B)$ are in $\mathcal{C}$, then there exists $\tau\in\mathcal{T}$ such that $\dom{\tau} = (\A,\Lip_\A)$ and $\codom{\tau} = (\B,\Lip_\B)$.
\item If $\tau = (\D,\Lip,\pi,\rho) \in\mathcal{T}$ then $\tau^{-1} = (\D,\Lip,\rho,\pi) \in \mathcal{T}$.
\item If there exists a full quantum isometry $h : (\A,\Lip_\A) \rightarrow (\B,\Lip_\B)$ for any $(\A,\Lip_\A),(\B,\Lip_\B) \in \mathcal{C}$, then the tunnels $(\A,\Lip_\A,\mathrm{id}_\A,h)$ and $(\B,\Lip_\B,h^{-1},\mathrm{id}_\B)$ are in $\mathcal{T}$, where $\mathrm{id}_{\alg{E}}$ is the identity automorphism of any C*-algebra $\alg{E}$.
\item If $\tau,\gamma \in \mathcal{T}$ with $\codom{\tau} = \dom{\gamma}$ them, for all $\varepsilon > 0$, there exists $\tau\circ_\varepsilon\gamma  \in \mathcal{T}$ such that $\dom{\tau} = \dom{\tau\circ_\varepsilon\gamma}$, $\codom{\gamma} = \codom{\tau\circ_\varepsilon\gamma}$ and
\begin{equation*}
\tunnelextent{\tau\circ_\varepsilon \gamma} < \tunnelextent{\tau} + \tunnelextent{\gamma} + \varepsilon \text{.}
\end{equation*}
\end{enumerate}
\end{definition}

The most important example of an appropriate class of tunnels is given by:
\begin{proposition}[{\cite{Latremoliere14}}]
If $F$ is an admissible function, then the class of all $F$-tunnels is appropriate for the class of all {\Qqcms{F}s}.
\end{proposition}

We now have all the ingredients to define our propinquity.

\begin{notation}
Let $\mathcal{T}$ be a class of tunnels appropriate for a nonempty class $\mathcal{C}$ of {\Qqcms{F}s} for some admissible function $F$. The class of all tunnels in $\mathcal{T}$ from $(\A,\Lip_\B)\in\mathcal{C}$ to $(\B,\Lip_\B)\in\mathcal{C}$ is denoted by
\begin{equation*}
\tunnelset{\A,\Lip_\A}{\B,\Lip_\B}{\mathcal{T}}\text{.}
\end{equation*}
\end{notation}

\begin{definition}\label{dual-propinquity-def}
Let $\mathcal{T}$ be a class of tunnels appropriate for a nonempty class $\mathcal{C}$ of {\Qqcms{F}s} for some admissible function $F$. Let $(\A,\Lip_\A)$ and $(\B,\Lip_\B)$ be in $\mathcal{C}$.

The \emph{dual $\mathcal{T}$-propinquity} $\dpropinquity{\mathcal{T}}((\A,\Lip_\A),(\B,\Lip_\B))$ between $(\A,\Lip_\A)$ and $(\B,\Lip_\B)$ is the real number
\begin{equation*}
\inf\left\{ \tunnelextent{\tau} : \tau \in \tunnelset{\A,\Lip_\A}{\B,\Lip_\B}{\mathcal{T}} \right\} \text{.}
\end{equation*}
\end{definition}

\begin{notation}
Let $F$ be some admissible function. If $\mathcal{C}$ is the class of all {\Qqcms{F}s} and $\mathcal{T}$ is the class of all $F$ tunnels, then the dual $\mathcal{T}$-propinquity is simply denoted by $\dpropinquity{F}$.

If, moreover, $F(x,y,z,w) = xw + yz$ for all $x,y,z,w \geq 0$, i.e. if we work within the class of {\Lqcms s}, then we denote $\dpropinquity{F}$ simply as $\dpropinquity{}$.
\end{notation}

The main properties of our metric are summarized in the following:

\begin{theorem}[{\cite{Latremoliere13c,Latremoliere14}}]
Let $F$ be an admissible function. If $\mathcal{T}$ is a class of tunnels admissible for a nonempty class $\mathcal{C}$ of {\Qqcms{F}s}, then the dual $\mathcal{T}$-propinquity $\dpropinquity{\mathcal{T}}$ is a metric up to full quantum isometry on $\mathcal{C}$.

When restricted to the class of classical metric spaces, through the identification given in Example (\ref{classical-ex}), the propinquity $\propinquity{F}$ induces the same topology as the Gromov-Hausdorff distance \cite{Edwards75,Gromov81}.

Moreover the dual propinquity $\dpropinquity{F}$ is complete if $F$ is continuous.
\end{theorem}

Tunnels behave, in some ways, like morphisms, albeit their composition is only defined up to an arbitrary choice of a positive constant. We will use the morphism-like properties of tunnels in a fundamental  way in this paper and so we now recall them. Lipschitz elements of a {\gQqcms} have images by tunnels, although these images are themselves subsets of the codomain.

\begin{definition}\label{targetset-def}
Let $F$ be an admissible function, and let $(\A,\Lip_\A)$ and $(\B,\Lip_\B)$ be two {\Qqcms{F}s}. Let $\tau = (\D, \Lip_\D, \pi_\A, \pi_\B)$ be an $F$-tunnel from $(\A,\Lip_\A)$ to $(\B,\Lip_\B)$. The \emph{$l$-target set} of $a\in\dom{\Lip_A}$ for $l\geq \Lip_\A(a)$ is
\begin{equation*}
\targetsettunnel{\tau}{a}{l} = \left\{ \pi_\B(d) \in \sa{\B} \middle\vert \begin{array}{l}
d\in \sa{\D} \\
\Lip_\D(d) \leq l \\
\pi_\A(d) = a
\end{array}  \right\} \text{.}
\end{equation*}
\end{definition}

Target sets are a form of set-valued morphisms:
\begin{proposition}[{\cite{Latremoliere13b,Latremoliere15}}]\label{targetset-morphism-prop}
Let $F$ be an admissible function, and let $(\A,\Lip_\A)$ and $(\B,\Lip_\B)$ be {\Qqcms{F}s}. Let $\tau$ be an $F$-tunnel from $(\A,\Lip_\A)$ to $(\B,\Lip_\B)$. For all $a,a'\in \dom{\Lip_\A}$ and $l\geq \max\{\Lip_\A(a),\Lip_\A(a')\}$, if $b\in\targetsettunnel{\tau}{a}{l}$ and $b'\in\targetsettunnel{\tau}{a'}{l}$ then:
\begin{enumerate}
\item for all $t \in \R$
\begin{equation*}
b + t b' \in \targetsettunnel{\tau}{a + t a'}{(1+|t|)l} \text{,}
\end{equation*}
\item if $K(a,a',l,\tau) = F(\|a\|_A + 2 l \tunnelextent{\tau}, \|a'\|_\A + 2 l \tunnelextent{\tau}, l, l)$ then
\begin{equation*}
 \Jordan{b}{b'} \in \targetsettunnel{\tau}{\Jordan{a}{a'}}{K(a,a',l,\tau)}\text{ and }\Lie{b}{b'} \in \targetsettunnel{\tau}{\Lie{a}{a'}}{K(a,a',l,\tau)}\text{.}
\end{equation*}
\end{enumerate} 
\end{proposition}

The observation on which many proofs, including proofs in this paper, rely, is that target set's size is controlled by the extent of the tunnel.

\begin{proposition}[{\cite{Latremoliere13b}}]\label{targetset-norm-prop}
Let $F$ be an admissible function and let $(\A,\Lip_\A)$ and $(\B,\Lip_\B)$ be {\Qqcms{F}s}. Let $\tau$ be an $F$-tunnel from $(\A,\Lip_\A)$ to $(\B,\Lip_\B)$. For all $a\in \dom{\Lip_\A}$ and $l\geq \Lip_\A(a)$, the target set $\targetsettunnel{\tau}{a}{l}$ is a nonempty compact subset of $\sa{\B}$, and moreover if $d\in\sa{\D}$ is chosen so that $\pi_\A(d) = a$ and $\Lip_\D(d) \leq l$ then
\begin{equation*}
\|d\|_\D \leq \|a\|_\A + l \tunnelextent{\tau}\text{, }
\end{equation*}
so in particular, if $b \in \targetsettunnel{\tau}{a}{l}$ then
\begin{equation*}
\|b\|_\B \leq \|a\|_\A + 2 l \tunnelextent{\tau} \text{.}
\end{equation*}
As a corollary, if $a,a' \in \dom{\Lip_\A}$ and $l\geq\max\{\Lip_\A(a),\Lip_\A(a')\}$ then
\begin{equation*}
\forall b\in\targetsettunnel{\tau}{a}{l} \forall b' \in \targetsettunnel{\tau}{a'}{l} \quad \|b - b'\|_\B \leq \|a - a'\|_\A + 2 l \tunnelextent{\tau}
\end{equation*}
and therefore
\begin{equation*}
\diam{\targetsettunnel{\tau}{a}{l}}{\|\cdot\|_\B} \leq 4 l \tunnelextent{\tau} \text{.}
\end{equation*}
\end{proposition}

We conclude this section with a new lemma. A tunnel enables us to pick corresponding $\varepsilon$-dense finite subsets of its domain and codomain. This lemma will prove helpful in checking that the morphisms which our main theorem builds will be non-trivial, under appropriate conditions.

\begin{lemma}\label{totally-bounded-correspondance-lemma}
Let $F$ be an admissible function. Let $(\A,\Lip_\A)$ and $(\B,\Lip_\B)$ be two {\Qqcms{F}s}. Let $\varepsilon > 0$ and let $\tau$ be an $F$-tunnel from $(\A,\Lip_\A)$ to $(\B,\Lip_\B)$ such that $\tunnelextent{\tau}<\varepsilon$. Let $E$ be a finite subset of $\{a \in \sa{\A} : \Lip_\A(a) \leq 1\}$ such that for all $a\in\sa{\A}$ with $\Lip_\A(a) \leq 1$, there exists $a' \in E$ and $t \in \R$ such that
\begin{equation*}
\|a - (a' + t\unit_\A)\|_\A < \varepsilon \text{.}
\end{equation*}

Then there exists a finite subset $G \subseteq \{b \in\sa{\B} : \Lip_\B(b) \leq 1\}$ such that:
\begin{enumerate}
\item for all $b\in\sa{\B}$ with $\Lip_\B(b)\leq 1$, there exists $t \in \R$ and $b' \in G$ such that $\|b - (b' + t\unit_\B)\|_\B \leq 3 \varepsilon$,
\item for all $b\in G$ there exists $a\in E$ such that $b \in \targetsettunnel{\tau}{a}{1}$, and conversely, for all $a\in E$ there exists $b\in G$ such that $a \in \targetsettunnel{\tau^{-1}}{b}{1}$.
\end{enumerate}
\end{lemma}

\begin{proof}
We write $\tau = (\D,\Lip_\D,\pi_\A,\pi_\B)$.

For all $a \in E$, choose $f(a) \in \targetsettunnel{\tau}{a}{1}$. We thus define a function $f : E \rightarrow \sa{\B}$. Note that $\Lip_\B(f(a)) \leq 1$ by construction. Moreover, note that $a\in\targetsettunnel{\tau^{-1}}{f(a)}{1}$.

Let $G = \{ f(a) : a \in E \}$. Now, let $b \in \sa{\B}$ with $\Lip_\B(b) \leq 1$. Let $a \in \targetsettunnel{\tau^{-1}}{b}{1}$. Since $\Lip_\A(a)\leq 1$, there exists $a' \in E$ and $t\in \R$ such that $\|a - (a' + t\unit_\A)\|_\A < \varepsilon$.

By Definition (\ref{targetset-def}), there exists $d\in \dom{\Lip_\D}$ with $\pi_\A(d) = a'$, $\pi_\B(d) = f(a')$, and $\Lip_\D(d) \leq 1$. Now, $\pi_\A(d + t\unit_\D) = a' + t\unit_\A$, while $\pi_\B(d + t \unit_\B) = f(a') + t\unit_\B$, and moreover
\begin{equation*}
\Lip_\D(d + t \unit_\D) = \Lip_\D(d) \leq 1 \text{.}
\end{equation*}

Therefore $f(a') + t\unit_\B \in \targetsettunnel{\tau}{a' + t \unit_\A}{1}$, again by Definition (\ref{targetset-def}). It follows by Proposition (\ref{targetset-norm-prop}) that
\begin{equation*}
\|b - (f(a') + t \unit_\B) \|_\B \leq \|a - (a' + t\unit_\A)\|_\A + 2 \tunnelextent{\tau} \leq 3\varepsilon \text{.}
\end{equation*} 

This concludes our proof.
\end{proof}

We now turn to our main theorem in the next section.

\section{Actions of Categories on Limits}

As tunnels behave much as set-valued morphisms, we construct a natural analogue to a commutative diagram, which lies at the core of our main theorem. Thus, for the next few results, we will work with the following ingredients.

\begin{hypothesis}\label{hypothesis-0}
Let $F$ be an admissible function. Let $(\A,\Lip_\A)$, $(\B,\Lip_\B)$, $(\alg{E},\Lip_{\alg{E}})$, and $(\alg{F},\Lip_{\alg{F}})$ be {\Qqcms{F}s}. Let $\tau$ be an $F$-tunnel from $(\A,\Lip_\A)$ to $(\alg{E},\Lip_{\alg{E}})$, and let $\gamma$ be an $F$-tunnel from $(\alg{F},\Lip_{\alg{F}})$ to $(\B,\Lip_\B)$. Let $\varphi : \alg{E} \rightarrow \alg{F}$ be a Lipschitz linear map and let $D > 0$ such that $\Lip_{\alg{F}}\circ\varphi \leq D \Lip_{\alg{E}}$ and $\opnorm{\varphi}{\alg{E}}{\alg{F}} \leq D$.
\end{hypothesis}

Let $\tau$ be a tunnel with domain $(\A,\Lip_\A)$. As we look at tunnels as a form of morphism and target sets as the mean to get an image for elements by such morphisms, it is natural to define the image of a set $A\subseteq\sa{\A}$ by $\tau$ by simply setting for any real $l > 0$
\begin{equation*}
\targetsettunnel{\tau}{A}{l} = \bigcup_{a \in A} \targetsettunnel{\tau}{a}{l}\text{.}
\end{equation*}
This image of a set is not empty as long as there exists an element $a\in A$ such that $\Lip_\A(a) \leq l$.

For our purpose, the following specific application of this idea will be central.

\begin{definition}\label{targetset-forward-def}
Given Hypothesis (\ref{hypothesis-0}), for $a \in \dom{\Lip_\A}$ and $l \geq \Lip_\A(a)$, the \emph{$l$-image-target set} of $a$ is defined by
\begin{equation*}
\targetsetimage{\varphi,\tau}{a}{l} = \varphi\left(\targetsettunnel{\tau}{a}{l}\right)
\end{equation*}
and the \emph{$l$-forward-target set} of $a$ is defined by
\begin{equation*}
\targetsetforward{\tau,\varphi,\gamma}{a}{l,D} = \targetsettunnel{\gamma}{\targetsetimage{\varphi,\tau}{a}{l}}{D l} \text{.}
\end{equation*}
\end{definition}

\begin{remark}
Using Hypothesis (\ref{hypothesis-0}), we thus have, for $a\in\dom{\Lip_\A}$ and $l\geq \Lip_\A(a)$
\begin{equation*}
\targetsetforward{\tau,\varphi,\gamma}{a}{l,D} = \bigcup_{c \in \targetsetimage{\varphi,\tau}{a}{l}} \targetsettunnel{\gamma}{c}{D l} = \bigcup_{b \in \targetsettunnel{\tau}{a}{l}} \targetsettunnel{\gamma}{\varphi(b)}{D l} \text{.}
\end{equation*}
\end{remark}

We thus form a sort of commutative diagram, though using set-valued functions

\begin{equation*}
\xymatrix{
a \in \sa{\A} \ar[d]_{\targetsettunnel{\tau}{\cdot}{l}} \ar[r] & \targetsetforward{\tau,\varphi,\gamma}{a}{l,D} \subseteq \sa{\B} \\
\targetsettunnel{\tau}{a}{l} \subseteq \sa{\alg{E}} \ar[r]_{\varphi} & \targetsetimage{\tau,\varphi}{a}{l} \subseteq \sa{\alg{F}} \ar[u]_{\targetsettunnel{\gamma}{\cdot}{D l}} 
}
\end{equation*}

By Definition (\ref{targetset-forward-def}), if $l' \geq l$ and $D' \geq D$, then with the notations of Hypothesis (\ref{hypothesis-0}), we have for $a\in\sa{\A}$:
\begin{equation*}
\targetsetforward{\tau,\varphi,\gamma}{a}{l,D} \subseteq \targetsetforward{\tau,\varphi,\gamma}{a}{l',D'} \text{.}
\end{equation*}

We now establish that the forward target sets enjoy properties akin to target sets. We begin with the linearity property.

\begin{lemma}\label{targetset-forward-linear-morphism-lemma}
Assume Hypothesis (\ref{hypothesis-0}). If
\begin{equation*}
  a,a' \in \dom{\Lip_\A}\text{ and }l \geq \max\{\Lip_\A(a), \Lip_\A(a')\}
\end{equation*}
then for all $t \in \R$, $b \in \targetsetforward{\tau,\varphi,\gamma}{a}{l,D}$ and $b' \in \targetsetforward{\tau,\varphi,\gamma}{a'}{l,D}$, we have
\begin{equation*}
b + t b' \in \targetsetforward{\tau,\varphi,\gamma}{a + t a'}{(1 + |t|)l,D} \text{.}
\end{equation*}
\end{lemma}

\begin{proof}
Let $b \in \targetsetforward{\tau,\varphi,\gamma}{a}{l,D}$ and $b' \in \targetsetforward{\tau,\varphi,\gamma}{a'}{l,D}$. By Definition (\ref{targetset-forward-def}), there exists $e \in \targetsettunnel{\tau}{a}{l}$ and $e' \in \targetsettunnel{\tau}{a'}{l}$ such that, if $f = \varphi(e)$ and $f' = \varphi(e')$, then $b \in \targetsettunnel{\gamma}{f}{D l}$ and $b' \in \targetsettunnel{\gamma}{f'}{Dl}$.

Now, $e + t e' \in \targetsettunnel{\tau}{a + t a'}{(1 + |t|)l}$ by Proposition (\ref{targetset-morphism-prop}). So
\begin{equation*}
f + t f' = \varphi(e + t e') \in \varphi\left(\targetsettunnel{\tau}{a + t a'}{(1+|t|)l} \right)\text{.}
\end{equation*}
In particular, $\Lip_{\alg{F}}(f + t f') \leq D\Lip_{\alg{E}}(e + t e') \leq D l (1+|t|)$.

On the other hand, again by Proposition (\ref{targetset-morphism-prop})
\begin{equation*}
b + t b' \in \targetsettunnel{\gamma}{f + t f'}{(1 + |t|)D l} \text{.}
\end{equation*}

Thus $b + t b' \in \targetsetforward{\tau,\varphi,\gamma}{a + t a'}{(1+|t|)l,D}$.
\end{proof}

We now establish the main geometric property of forward-target sets --- namely, their diameters in norm is controlled by the extent of the tunnels from which they are defined.

\begin{lemma}\label{targetset-forward-norm-lemma}
Assume Hypothesis (\ref{hypothesis-0}). Let
\begin{equation*}
  a,a' \in \dom{\Lip_\A}\text{ and }l \geq \max\{\Lip_\A(a), \Lip_\A(a')\}\text{.}
\end{equation*}
 If $f \in \targetsetimage{\tau,\varphi}{a}{l}$ then
\begin{equation*}
\left\| f \right\|_{\alg{F}} \leq D\left(\|a\|_\A + 2l \tunnelextent{\tau} \right)
\end{equation*}
and therefore for all $f' \in \targetsetimage{\tau,\varphi}{a'}{l}$
\begin{equation*}
\left\| f - f' \right\|_{\alg{F}} \leq D\left(\|a - a'\|_\A + 4l \tunnelextent{\tau} \right) \text{.}
\end{equation*}

If moreover $b \in \targetsetforward{\tau,\varphi,\gamma}{a}{l,D}$, then
\begin{equation*}
\left\|b\right\|_\B \leq D \left( \|a\|_\A + 2 l \left(\tunnelextent{\tau} + \tunnelextent{\gamma}\right) \right)
\end{equation*}
an therefore for all $b' \in \targetsetforward{\tau,\varphi,\gamma}{a'}{l,D}$
\begin{equation*}
\left\|b - b'\right\|_\B \leq D \left( \|a - a'\|_\A + 4 l \left(\tunnelextent{\tau} + \tunnelextent{\gamma}\right) \right)
\end{equation*}
so that, in particular
\begin{equation*}
\diam{\targetsetforward{\tau,\varphi,\gamma}{a}{l,D}}{\|\cdot\|_\B} \leq 4 D l \left(\tunnelextent{\tau} + \tunnelextent{\gamma}\right)\text{.}
\end{equation*}
\end{lemma}

\begin{proof}
Let $f \in \targetsetimage{\tau,\varphi}{a}{l}$. There exists, by Definition (\ref{targetset-forward-def}), an element $e \in \targetsettunnel{\tau}{a}{l}$ such that $f = \varphi(e)$. Now by Proposition (\ref{targetset-norm-prop}), we have
\begin{equation}\label{diam-eq-1}
\|f\|_{\alg{F}} \leq D \|e\|_{\alg{E}} \leq D(\|a\|_\A + 2 l \tunnelextent{\tau}) \text{,}
\end{equation}
as desired. If $f' \in \targetsetimage{\tau,\varphi}{a'}{l}$ then $f' = \varphi(e')$ with $e' \in \targetsettunnel{\tau}{a}{l}$, and thus $f - f' = \varphi(e-e')$ with $e-e' \in \targetsettunnel{\tau}{a-a'}{2l}$ by Proposition (\ref{targetset-morphism-prop}). We conclude
\begin{equation*}
\left\| f - f' \right\|_{\alg{F}} \leq D\left(\|a - a'\|_\A + 4l\tunnelextent{\tau}\right)
\end{equation*}
using Expression (\ref{diam-eq-1}).

Let now $b \in \targetsetforward{\tau,\varphi,\gamma}{a}{l,D}$. By Definition (\ref{targetset-forward-def}), there exists $f \in \targetsetimage{\tau,\varphi}{a}{l}$ such that $b \in \targetsettunnel{\gamma}{f}{D l}$. By Proposition (\ref{targetset-norm-prop}) and Expression (\ref{diam-eq-1}), we thus have
\begin{equation}\label{targetset-forward-norm-eq1}
\|b\|_\B \leq \|f\|_{\alg{F}} + 2 D l \tunnelextent{\gamma} \leq D(\|a\|_\A + 2 l (\tunnelextent{\tau} + \tunnelextent{\gamma}) )\text{.}
\end{equation}

Let $b' \in \targetsetforward{\tau,\varphi,\gamma}{a'}{l,D}$. By Lemma (\ref{targetset-forward-linear-morphism-lemma}), we have
\begin{equation*}
b - b' \in \targetsetforward{\tau,\varphi,\gamma}{a - a'}{2l,D}
\end{equation*}
and therefore, by Expression (\ref{targetset-forward-norm-eq1})
\begin{equation}\label{targetset-forward-norm-eq2}
\left\|b - b'\right\|_\B \leq D \left( \|a - a'\|_\A + 4 l (\tunnelextent{\tau} + \tunnelextent{\gamma}) \right)
\end{equation}
as needed.

Setting $a=a'$ in Expression (\ref{targetset-forward-norm-eq2}), we then get
\begin{equation*}
\diam{\targetsetforward{\tau,\varphi,\gamma}{a}{l,D}}{\|\cdot\|_\B} \leq 4 D l (\tunnelextent{\tau} + \tunnelextent{\gamma})\text{,}
\end{equation*}
thus concluding our lemma.
\end{proof}

We now complete our list of algebraic properties of the forward target sets by proving the following lemma.

\begin{lemma}\label{targetset-forward-JL-morphism-lemma}
Assume Hypothesis (\ref{hypothesis-0}). If 
\begin{equation*}
  a,a' \in \dom{\Lip_\A}\text{ and }l \geq \max\{\Lip_\A(a), \Lip_\A(a')\}\text{, }
\end{equation*}
 if $b\in\targetsetforward{\tau,\varphi,\gamma}{a}{l, D}$ and $b' \in \targetsetforward{\tau,\varphi,\gamma}{a'}{l, D}$, and if moreover $\varphi$ is a Jordan-Lie morphism, then, setting
\begin{multline*}
M_{\tau,\gamma}(\|a\|_\A,\|a'\|_\A,l,D) = \max\bigg\{ F\left(\|a\|_\A + 2 l \tunnelextent{\tau}, \|a'\|_\A + 2 l \tunnelextent{\tau}, l, l\right), \\ 
\frac{1}{D} F\left( D\left(\|a\|_\A + 2 l \tunnelextent{\tau} + 2 l \tunnelextent{\gamma}\right), D\left(\|a'\|_\A + 2 l \tunnelextent{\tau} + 2 l \tunnelextent{\gamma}\right), D l, D l  \right)  \bigg\}
\end{multline*}
we conclude
\begin{equation*}
\Jordan{b}{b'} \in \targetsetforward{\tau,\varphi,\gamma}{\Jordan{a}{a'}}{M_{\tau,\gamma}(\|a\|_\A,\|a'\|_\A,l,D),D}
\end{equation*}
and
\begin{equation*}
\Lie{b}{b'} \in \targetsetforward{\tau,\varphi,\gamma}{\Lie{a}{a'}}{M_{\tau,\gamma}(\|a\|_\A,\|a'\|_\A,l,D),D}\text{.}
\end{equation*}
\end{lemma}

\begin{proof}
By Definition (\ref{targetset-forward-def}), there exists $e \in \targetsettunnel{\tau}{a}{l}$ and $e'\in\targetsettunnel{\tau}{a'}{l}$ such that, if $f = \varphi(e)$ and $f' = \varphi(e')$, then $b \in \targetsettunnel{\gamma}{f}{D l}$ and $b' \in \targetsettunnel{\gamma}{f'}{D l}$.

Assume $\varphi$ is a Lie-Jordan morphism. We then have by Proposition (\ref{targetset-morphism-prop})
\begin{equation*}
\Jordan{e}{e'} \in \targetsettunnel{\tau}{\Jordan{a}{a'}}{F\left(\|a\|_\A + 2 l \tunnelextent{\tau}, \|a'\|_\A + 2 l \tunnelextent{\tau}, l, l\right)}
\end{equation*}
and $\Jordan{f}{f'} = \Jordan{\varphi(e)}{\varphi(e')} = \varphi(\Jordan{e}{e'})$ so
\begin{equation*}
\Lip_{\alg{F}}\left(\Jordan{f}{f'}\right) \leq D F\left(\|a\|_\A + 2 l \tunnelextent{\tau}, \|a'\|_\A + 2 l \tunnelextent{\tau}, l, l\right) \leq D M_{\tau,\gamma}(\|a\|_\A,\|a'\|_\A,l,D)\text{.}
\end{equation*}

We also have, again by Proposition (\ref{targetset-morphism-prop})
\begin{equation*}
\Jordan{b}{b'} \in \targetsettunnel{\gamma}{\Jordan{f}{f'}}{F\left(\|f\|_{\alg{F}} + 2 D l \tunnelextent{\gamma}, \|f'\|_{\alg{F}} + 2 D l \tunnelextent{\gamma}, D l, D l \right)}
\end{equation*}
with
\begin{equation*}
\|f\|_{\alg{F}} \leq D\|e\|_{\alg{E}} \leq D(\|a\|_\A + 2 l \tunnelextent{\tau})
\end{equation*}
and
\begin{equation*}
\|f'\|_{\alg{F}} \leq D\|e'\|_{\alg{E}} \leq D(\|a'\|_\A + 2 l \tunnelextent{\tau})\text{, }
\end{equation*}
by Lemma (\ref{targetset-forward-norm-lemma}), so
\begin{multline*}
F\left(\|f\|_{\alg{F}} + 2 D l \tunnelextent{\tau}, \|f'\|_{\alg{F}} + 2 D l \tunnelextent{\tau}, D l, D l \right) \leq \\ F\left(D(\|a\|_\A + 2l\tunnelextent{\tau}) + 2 D l \tunnelextent{\tau}, D(\|a'\|_\A + 2 l \tunnelextent{\tau})+ 2 D l \tunnelextent{\tau}, D l, D l \right) \\ \leq D M_{\tau,\gamma}(\|a\|_\A,\|a'\|_\A,l,D) \text{.}
\end{multline*}
By Definition, we thus have
$\Jordan{b}{b'} \in \targetsetforward{\tau,\varphi,\gamma}{\Jordan{a}{a'}}{M_{\tau,\gamma}(\|a\|_\A,\|a'\|_\A,l,D),D}$. The same proof holds for the Lie product.
\end{proof}

Last, we identify the topological nature of forward target sets.

\begin{lemma}\label{targetset-forward-compact-lemma}
Assume Hypothesis (\ref{hypothesis-0}). For all $a\in\dom{\Lip_\A}$ and $l \geq \Lip_\A(a)$, the sets $\targetsetimage{\tau,\varphi}{a}{l}$ and $\targetsetforward{\tau,\varphi,\gamma}{a}{l,D}$ are nonempty and compact, respectively, in $\sa{\alg{F}}$ and $\sa{\B}$.
\end{lemma}

\begin{proof}
By \cite[Lemma 4.2]{Latremoliere13b}, the set $\targetsettunnel{\tau}{a}{l}$ is a nonempty compact subset of $\sa{\alg{E}}$. Since $\varphi$ is continuous, $\targetsetimage{\tau,\varphi}{a}{l}$ is not empty and compact in $\sa{\alg{F}}$ as well. 

Let now $e \in \targetsettunnel{\tau}{a}{l}$. By construction, $\Lip_{\alg{F}}\circ\varphi(e) \leq D \Lip_{\alg{E}}(e) \leq D l$. Hence again by \cite[Lemma 4.2]{Latremoliere13b}, the set $\targetsettunnel{\gamma}{\varphi(e)}{D l}$ is a nonempty compact set. In particular, $\targetsetforward{\tau,\varphi,\gamma}{a}{l,D}$ is not empty as well. Moreover, by Lemma (\ref{targetset-forward-norm-lemma}), the following inclusion holds:
\begin{equation*}
\targetsetforward{\tau,\varphi,\gamma}{a}{l,D} \subseteq \left\{ b \in \sa{\B} : \Lip_\B(b) \leq D l, \|b\|_\A \leq D\left(\|a\|_\A + 2 l \left(\tunnelextent{\tau} + \tunnelextent{\gamma}\right)\right) \right\}\text{,}
\end{equation*}
and the set on the right hand-side is compact since $\Lip$ is an L-seminorm. So $\targetsetforward{\tau,\varphi,\gamma}{a}{l,D}$ is totally bounded in $\sa{\B}$.

To prove that $\targetsetforward{\tau,\varphi,\gamma}{a}{l,D}$ is compact, it is thus sufficient to show that is is closed, since it is totally bounded and $\sa{\B}$ is complete. To this end, we write the tunnel $\gamma$ as $(\D,\Lip_{\D},\pi,\rho)$.

Let $(b_k)_{k\in\N}$ be a sequence in $\targetsetforward{\tau,\varphi,\gamma}{a}{l,D}$, converging in $\B$ to some $b$. By Definition (\ref{targetset-forward-def}), for each $k\in\N$, there exists $f_k \in \varphi\left(\targetsettunnel{\tau}{a}{l}\right)$ such that $b_k \in \targetsettunnel{\gamma}{f_k}{D l}$. There exists $e_k \in \targetsettunnel{\tau}{a}{l}$ such that $f_k = \varphi(e_k)$. Now, $\targetsettunnel{\tau}{a}{l}$ is compact, so there exists a convergent subsequence $(e_{j(k)})_{k\in\N}$ converging in norm to some $e \in \targetsettunnel{\tau}{a}{l}$. Therefore, $(\varphi(e_{j(k)}))_{k\in\N}$ converges to $\varphi(e)$ as $\varphi$ is continuous. We denote $\varphi(e)$ by $f$. Thus $f \in \targetsetimage{\tau,\varphi}{a}{l}$.

Now, for each $k\in\N$, since $b_{j(k)} \in \targetsettunnel{\gamma}{f_{j(k)}}{D l}$, there exists $d_k \in \sa{\D}$ such that $\Lip_{\D}(d_k) \leq D l$ while $\pi(d_k) = f_{j(k)}$ and $\rho(d_k) = b_{j(k)}$.

Once more, $(d_k)_{k\in\N}$ lies in the compact set
\begin{equation*}
\left\{ w \in \sa{\D} : \Lip_{\D}(w) \leq D l, \|w\|_{\D} \leq D(\|a\|_\A + 2 l (\tunnelextent{\tau} + \tunnelextent{\gamma}))  \right\}
\end{equation*}
and thus we extract a convergent subsequence $(d_{m(k)})_{k\in\N}$ with limit $d$ with $\Lip_{\D}(d) \leq D l$.

Moreover by continuity
\begin{equation*}
\lim_{k\rightarrow\infty} \rho(d_{m(k)}) = \lim_{k\rightarrow\infty} b_{j\circ m(k)} = b
\text{ and }
\lim_{k\rightarrow\infty} \pi(d_{m(k)}) = \lim_{k\rightarrow\infty} f_{j\circ m(k)} = f \text{.}
\end{equation*}

By Definition, we thus conclude that $b \in \targetsettunnel{\gamma}{f}{D l}$ and therefore $b \in \targetsetforward{\tau,\varphi,\gamma}{a}{l,D}$. Thus $\targetsetforward{\tau,\varphi,\gamma}{a}{l,D}$ is closed. This concludes our lemma.
\end{proof}

We now provide the setting for our main theorem. We will prove not only that families of morphisms pass to the limit, informally speaking, for the propinquity, but also that composition of morphisms also passes to the limit. We encode the structure of composition with the natural notion of a (small) category. We provide below the well-known definition of a small category and of a functor to fix our notations.

\begin{definition}
A \emph{small category} $(\mathcal{S},\mathcal{S}^{(0)},d,c,\circ)$ is a pair of sets $\mathcal{S}$ and $\mathcal{S}^{(0)}$, two maps $c,d : \mathcal{S} \rightarrow \mathcal{S}^{(0)}$ and a map $\circ  : \mathcal{S}^{(2)} \rightarrow \mathcal{S}$ where
\begin{equation*}
\mathcal{S}^{(2)} = \left\{ (s,s') \in \mathcal{S}^2 : d(s) = c(s') \right\}\text{,}
\end{equation*}
such that
\begin{enumerate}
\item for all $(s,s') \in \mathcal{S}^{(2)}$, we have $d(s\circ s') = d(s')$ and $c(s\circ s') = c(s)$,
\item if $s,s',s'' \in \mathcal{S}$ with $(s,s'), (s',s'') \in \mathcal{S}^{(2)}$, then $(s \circ s') \circ s'' = s \circ (s' \circ s'')$,
\item for all $t\in\mathcal{S}^{(0)}$, there exists $i \in \mathcal{S}$ such that $c(i)=d(i)=t$ and for all $s \in \mathcal{S}$, if $c(s)=t$ we have $i\circ s = s$, and if $d(s)=t$ then $s \circ i = s$. This map  $i$, which is necessarily unique, is called the identity of $t$.
\end{enumerate}
Elements of $\mathcal{S}$ are called morphisms, while elements of $\mathcal{S}^{(0)}$ are called objects. The map $d$ is the domain map, while $c$ is the codomain map, and $\circ$ is the composition of morphisms. Ordered pairs in $\mathcal{S}^{(2)}$ are called composable pairs.
\end{definition}

A morphism between categories is of course a functor:

\begin{definition}\label{functor-def}
A \emph{functor} $F$ from a category $(\mathcal{S},\mathcal{S}^{(0)},d,c,\circ)$ to a category $(\mathcal{C},\mathcal{C}^{(0)},s,t,\square)$ is a pair of maps (which by abuse of notations, we still denote as $F$), one from $\mathcal{S}^{(0)}$ to $\mathcal{C}^{(0)}$, and one from $\mathcal{S}$ to $\mathcal{C}$, such that:
\begin{enumerate}
\item $\forall \varphi \in \mathcal{S} \quad F(d(\varphi)) = s(F(\varphi)) \;\text{ and }\; F(c(\varphi)) = t(F(\varphi))$,
\item $\forall (\varphi,\varphi') \in \mathcal{S}^{(2)} \quad F(\varphi\circ \varphi') = F(\varphi)\square F(\varphi') \text{,}$
\item if $s\in\mathcal{S}$ is an identity element then $F(s)$ is also an identity element.
\end{enumerate}
\end{definition}

We note that Definition (\ref{functor-def}) is meaningful, as we easily check that if $(s,s') \in \mathcal{S}^{(2)}$ then $(F(s),F(s'))\in\mathcal{C}^{(2)}$.

Among morphisms in a category, isomorphisms, or invertible morphisms, will play an important role in our work.
\begin{definition}
Let $(\mathcal{S},\mathcal{S}^{(0)},d,c,\circ)$ be a small category. An element $u \in \mathcal{S}$ is \emph{invertible} (or an \emph{isomorphism}) when there exists $u'\in\mathcal{S}$, such that:
\begin{enumerate}
\item $d(u) = c(u')$ and $d(u') = c(u)$,
\item $u\circ u'$ is an identity of $c(u)$,
\item $u'\circ u$  is an identity of $d(u)$.
\end{enumerate}
The element $u'$ is called an inverse of $u$.
\end{definition}

Of course, there is at most one inverse of a morphism, and functors map invertible elements to invertible elements, and the image of the inverse of a morphism is the inverse of the image of said morphism.

A small category whose morphisms are all invertible is a groupoid. A groupoid with a single object is a group. Moreover, a small category with a single object is a monoid (a semigroup with an identity element). These are the algebraic structures to which we wish to apply our work in this paper, though our main results are proven at the more general level of small categories.

\medskip

We now state and prove our main theorem. As it has a long statement, we first describe its basic hypothesis. We basically start from  a small category with countable sets of objects and morphisms, and we assume given a sequence of ``almost-functors'' from that category to the category of quantum compact metric spaces (or a generalized version). We measure the defect in the functorial properties using our various norms and distances. Our  main theorem will show how this hypothesis suffices to prove the existence of a functor with desirable properties.

\begin{hypothesis}\label{main-thm-hyp}
Let $F$ be an admissible function. Let $\mathcal{T}$ be a class of tunnels appropriate for a nonempty class $\mathcal{C}$ of {\Qqcms{F}s}. 

Let $\left(\mathcal{S},\mathcal{S}^{(0)},d,c,\circ\right)$ be a small category, where $\mathcal{S}$ and $\mathcal{S}^{(0)}$ are both countable sets.

For all $s\in\mathcal{S}^{(0)}$, let $(\A^s,\Lip^s)\in\mathcal{C}$, and let $(\A_n^s,\Lip_n^s)_{n\in\N}$ be a sequence in $\mathcal{C}$ such that
\begin{equation*}
\lim_{n\rightarrow\infty} \dpropinquity{\mathcal{T}}((\A_n^s,\Lip_n^s), (\A^s,\Lip^s)) = 0 \text{.}
\end{equation*}

We set a few notations. Fix $s \in \mathcal{S}^{(0)}$. The operator norms of $\A^{d(s)}$ and $\A^{c(s)}$ are simply denoted by $\norm{\cdot}{d(s)}$ and $\norm{\cdot}{c(s)}$,  respectively. The operator norm for linear maps from $\A^{d(s)}$ to $\A^{c(s)}$ is denoted by $\opnorm{\cdot}{s}{}$, while the Monge-Kantorovich between the same linear maps is denoted by $\KantorovichDist{s}{}{\cdot}{\cdot}$. 

Similarly, for each $n\in\N$, the norms on $\A^{d(s)}_n$ and $\A^{c(s)}_n$ are denoted by $\norm{\cdot}{d(s),n}$ and $\norm{\cdot}{c(s),n}$, respectively. The operator norm for linear maps from $\A_n^{d(s)}$ to $\A_n^{c(s)}$ is denoted by $\opnorm{\cdot}{s}{n}$, while the Monge-Kantorovich between the same linear maps is denoted by $\KantorovichDist{s}{n}{\cdot}{\cdot}$.

Let $D : \mathcal{S} \rightarrow [0,\infty)$. 

We assume given, for all $s\in\mathcal{S}$ and for all $n\in\N$, a Lipschitz linear map
\begin{equation*}
\varphi_n^s : \left(\A_n^{d(s)},\Lip_n^{d(s)}\right) \rightarrow \left(\A_n^{c(s)},\Lip_n^{c(s)}\right)\text{.}
\end{equation*}

For each $s\in\mathcal{S}$ and $n\in\N$, we assume
\begin{equation*}
\opnorm{\varphi_n^s}{s}{n} \leq D(s) \text{ and }\Lip_n^{c(s)}\circ\varphi_n^s \leq D(s) \Lip_n^{d(s)} \text{,}
\end{equation*}
while for all $(s,t) \in \mathcal{S}^{(2)}$
\begin{equation*}
\lim_{n\rightarrow\infty} \KantorovichDist{s}{n}{\varphi_n^s\circ\varphi_n^t}{\varphi_n^{s\circ t}} = 0\text{.}
\end{equation*} 

Moreover, if $u \in \mathcal{S}$ is a unit, then we also assume
\begin{equation*}
  \lim_{n\rightarrow\infty} \KantorovichDist{u}{n}{\varphi_n^u}{\mathrm{id}_n^u} = 0\text{,}
\end{equation*}
with $\mathrm{id}_n^u$ the identity of $\A_n^{d(u)}$.
\end{hypothesis}

We will use repeatedly the following simple result, which we record as a lemma to help clarify our main argument.

\begin{lemma}\label{inclusion-lemma}
If $(A_n)_{n\in\N}$ and $(B_n)_{n\in\N}$ are two sequences of nonempty compact subsets of a complete metric space $(X,d)$ such that:
\begin{enumerate}
\item $\lim_{n\rightarrow\infty} \diam{B_n}{d}  = 0$,
\item for all $n\in\N$, we have $A_n\subseteq B_n$,
\end{enumerate}
then the sequence $(A_n)_{n\in\N}$ converges for $\Haus{d}$ if and only if $(B_n)_{n\in\N}$ does, in which case they have the same limit, a singleton of $X$.
\end{lemma}

\begin{proof}
We begin by proving that if a sequence $(C_n)_{n\in\N}$ of nonempty closed subsets of $X$ with $\lim_{n\rightarrow\infty} \diam{C_n}{d} = 0$ converges for $\Haus{d}$ then its limit $C$ is a singleton. Since $\diam{\cdot}{d}$ is a continuous function for $\Haus{d}$, it is immediate that $\diam{C}{d} = 0$. On the other hand, for each $n\in\N$, let $c_n \in C_n$. Let $\varepsilon > 0$. There exists $N\in\N$ such that for all $p,q\geq N$, we have $\Haus{d}(C_p,C_q) < \frac{\varepsilon}{3}$ and $\diam{C_n}{d} < \frac{\varepsilon}{3}$, so
\begin{equation*}
d(c_p,c_q) \leq \diam{C_p}{d} + \Haus{d}(C_p,C_q) + \diam{C_q}{d} < \varepsilon \text{,}
\end{equation*}
and thus $(c_n)_{n\in\N}$ is a  Cauchy sequence; since $(X,d)$ is complete, $(c_n)_{n\in\N}$ converges to some $c\in X$. Moreover, $d(c, C_n) \leq \varepsilon$ for $n\geq N$ by continuity. On the other hand, if $x \in C_n$ then $d(x,c) < \frac{4 \varepsilon}{3}$ for $n\geq N$. Hence $\Haus{d}(C_n,\{c\}) < \frac{4\varepsilon}{3}$ for $n\geq N$ and thus $C = \{c\}$.

We now prove the rest of our lemma.

Let $\varepsilon > 0$. There exists $N\in\N$ such that for all $n\geq N$
\begin{equation*}
\diam{B_n}{d} < \frac{\varepsilon}{2}\text{.}
\end{equation*}

Since $A_n\subseteq B_n$ for all $n\in\N$, we conclude that $\Haus{d}(A_n,B_n) < \frac{\varepsilon}{2}$ for all $n\geq N$.

Now, if $(A_n)_{n\in\N}$ converges for $\Haus{d}$, and $\{ a \}$ is its limit, then there exists $M \in \N$ such that for all $n\geq M$, we have
\begin{equation*}
\Haus{d}(A_n,\{a\}) < \frac{\varepsilon}{2}\text{.}
\end{equation*}

Therefore, for all $n\geq \max\{N,M\}$
\begin{equation*}
\Haus{d}(B_n,\{a\}) \leq \Haus{d}(B_n,A_n) + \Haus{d}(A_n,\{a\}) < \varepsilon\text{,}
\end{equation*}
and thus $(B_n)_{n\in\N}$ converges for $\Haus{d}$ to the same limit $\{a\}$.

If instead, $(B_n)_{n\in\N}$ converges to $\{b\}$ for $\Haus{d}$, then similarly, there exists $M\in\N$ such that $\Haus{d}(B_n,\{b\}) < \frac{\varepsilon}{2}$ and thus for all $n\geq\max\{N,M\}$, we conclude
\begin{equation*}
\Haus{d}(A_n,\{b\}) < \varepsilon
\end{equation*}
thus concluding our lemma.
\end{proof}

We now establish:

\begin{theorem}\label{main-thm}
If Hypothesis (\ref{main-thm-hyp}) holds, then there exist
\begin{itemize}
\item a functor $\Psi$ from $\mathfrak{S}$ to the category whose objects are the elements of $\mathcal{C}$, and whose morphisms are the Lipschitz linear maps,
\item a strictly increasing $j : \N \rightarrow \N$,
\item for all $s\in\mathcal{S}^{(0)}$ and $n\in\N$, a tunnel $\tau_n^s$ from $(\A^s,\Lip^s)$ to $(\A_n^s,\Lip_n^s)$,
\end{itemize}
 such that:
\begin{itemize}
  \item  $\lim_{n\rightarrow \infty} \tunnelextent{\tau_n^s} = 0$,
  \item $\forall s \in \mathcal{S}^{(0)} \quad \Psi(s) = (\A^s,\Lip^s)$
  \item we have:
    \begin{equation*}
      \forall s \in \mathcal{S} \quad \opnorm{\Psi(s)}{s}{} \leq \sqrt{2} D(s) \text{ and }\Lip^{c(s)}\circ\Psi(s) \leq D(s) \Lip^{d(s)} \text{,}
    \end{equation*}
  \item for all $s\in\mathcal{S}$ and $n\in\N$, if we write $\eta_n^s$ for the inverse of the tunnel $\tau_{n}^{c(s)}$, and we write $\sigma_n^s$ for $\tau_{n}^{d(s)}$, and if we write, for all $a\in\sa{\A^{d(s)}}$, and for all $D \geq D(s)$ and $l\geq \Lip^{d(s)}(a)$: 
    \begin{equation*}
      \targetsetforward{s,n}{a}{l,D} = \targetsettunnel{\eta_n^s}{\varphi_{n}^s\left( \targetsettunnel{\sigma_n^s}{a}{l}   \right)}{Dl}
    \end{equation*}
    i.e. $\targetsetforward{s,n}{a}{l,D}$ is the forward target set $\targetsetforward{\sigma_n^s,\varphi_{n}^s,\eta_n^s}{a}{l,D}$ defined in Definition (\ref{targetset-forward-def}), we the conclude:
    \begin{multline*}
      \forall s \in \mathcal{S} \quad \forall a \in \dom{\Lip^{d(s)}} \quad \forall l \geq \Lip^{d(s)}(a), D \geq D(s) \\
      \Haus{\|\cdot\|_{{c(s)}}}\left(\targetsetforward{s,j(n)}{a}{l,D}, \{\Psi(s)(a)\}\right) \xrightarrow{n\rightarrow\infty} 0 \text{.}
    \end{multline*}
  \end{itemize}
Furthermore:
\begin{enumerate}
\item If, for some $s\in\mathcal{S}$ and for all $n\in\N$, the map $\varphi_n^s$ is positive, and $\{a\in\dom{\Lip^{d(s)}}:a\geq 0\}$ is dense in $\{a\in\sa{\A^{d(s)}}: a \geq 0\}$, then $\Psi(s)$ is positive as well.

\item If, for some $s\in \mathcal{S}$, we can choose $(\varphi^s_n)_{n\in\N}$ such that $\varphi^s_n$ is a Lipschitz morphism for all $n\in\N$, then $\Psi(s)$ is a Lipschitz morphism --- in particular, $\opnorm{\Psi(s)}{s}{}  = 1$.

\item If, for some $s,s'\in \mathcal{S}$, we have $c(s) = c(s')$ and $d(s) = d(s')$, then $\KantorovichDist{s}{}{\Psi(s)}{\Psi(s')}$ is a limit point of the sequence $\left(\KantorovichDist{s}{n}{\varphi_n^s}{\varphi_n^{s'}}\right)_{n\in\N}$, and so in particular
\begin{equation*}
\liminf_{n\rightarrow\infty} \KantorovichDist{\Lip^{d(s)}_n}{\A^{c(s)}_n}{\varphi_n^s}{\varphi_n^{s'}} \leq \KantorovichDist{\Lip^{d(s)}}{\A^{c(s)}}{\varphi^s}{\varphi^{s'}} \leq  \limsup_{n\rightarrow\infty} \KantorovichDist{\Lip_n^{d(s)}}{\A^{c(s)}_n}{\varphi_n^s}{\varphi_n^{s'}}\text{.}
\end{equation*}

Moreover, if for some $s\in \mathcal{S}$, we have $d(s) = c(s)$, then
\begin{equation*}
\liminf_{n\rightarrow\infty} \KantorovichLength{\Lip_n^{d(s)}}(\varphi_n^s) \leq \KantorovichLength{\Lip^{d(s)}}(\Psi(s)) \leq \limsup_{n\rightarrow\infty} \KantorovichLength{\Lip_n^{d(s)}}(\varphi_n^s)\text{.}
\end{equation*}



\end{enumerate}

\end{theorem}

\begin{proof}
We will employ all the notation given in the assumptions of our theorem.

Up to extracting subsequences, a standard diagonal argument shows that since $\mathcal{S}^{(0)}$ is countable, we may as well assume that
\begin{equation*}
\forall s\in\mathcal{S}^{(0)} \; \forall n\in\N \quad \dpropinquity{\mathcal{T}}\left((\A_n^s,\Lip_n^s), (\A^s,\Lip^s)\right) \leq \frac{1}{n+2}\text{.}
\end{equation*}

For each $s\in\mathcal{S}^{(0)}$, $n\in\N$, let $\tau_n^s\in\mathcal{T}$ be a tunnel from $(\A^s,\Lip^s)$ to $(\A_n^s,\Lip_n^s)$ with $\tunnelextent{\tau_n^s} \leq \frac{1}{n+1}$. Let $\gamma_n\in\mathcal{T}$ be the inverse tunnel of $\tau_n^s$.

For any $s\in\mathcal{S}^{(0)}$, let $\alg{S}^s$ be a countable, dense subset of $\sa{\A^s}$ with $\alg{S}^s\subseteq\dom{\Lip^s}$ (such a set exists since $\sa{\A^s}$ is separable, as the  weak* topology of $\StateSpace(\A^s)$ is metrizable by definition of a {\qcms}).

We now prove our theorem in several steps. The first few steps prove the existence of a linear map which will eventually be denoted by $\Psi(s)$ for each $s\in\mathcal{S}$. The construction is of intrinsic interest, as it may be possible to use it to extract additional information on these maps; some of the later steps will be examples of this idea.

\begin{step}
Let $s\in\mathcal{S}$. Let $j_0 : \N \rightarrow \N$ be a strictly increasing function. Let $a \in \dom{\Lip^{d(s)}}$. There exists $j_1  : \N \rightarrow \N$, strictly increasing, and $f \in \sa{\A^{c(s)}}$, such that:
\begin{equation*}
\Lip^{c(s)}(f) \leq D(s) \Lip^{d(s)}(a) \text{ and }\|f\|_{c(s)} \leq D(s) \|a\|_{{d(s)}} \text{, }
\end{equation*}
and for all $l\geq \Lip^{d(s)}(a)$, the sequence $(\targetsetforward{s,j_0\circ  j_1(n)}{a}{l,D(s)})_{n\in\N}$ converges, for $\Haus{\|\cdot\|_{{c(s)}}}$, to the singleton $\{ f \}$.
\end{step}

Write $D = D(s)$ for this step. Let $a \in \dom{\Lip^{d(s)}}$. By Lemmas (\ref{targetset-forward-norm-lemma}) and (\ref{targetset-forward-compact-lemma}), the sequence
\begin{equation*}
\left(\targetsetforward{s,j_0(n)}{a}{\Lip^{d(s)}(a),D}\right)_{n\in\N}
\end{equation*}
is a sequence of nonempty compact subsets of
\begin{equation*}
\left\{ b \in \sa{\A^{c(s)}} : \Lip^{c(s)}(b) \leq D \Lip^{d(s)}(a)\text{ and }\|b\|_{{c(s)}} \leq D(\|a\|_{{d(s)}} + 2 \Lip^{d(s)}(a)) \right\}
\end{equation*}
which is itself compact in $\sa{\A^{c(s)}}$ since $\Lip^{c(s)}$ is an L-seminorm. Thus, there exists a strictly increasing function $j_1$ such that $\left(\targetsetforward{s,j_0\circ j_1(n)}{a}{\Lip^{d(s)}(a),D}\right)_{n\in\N}$ converges for the Hausdorff distance $\Haus{\|\cdot\|_{{c(s)}}}$, since the Hausdorff distance induces a compact topology on the hyperspace of closed subsets of a compact metric spaces (Blaschke's theorem \cite[Theorem 7.2.8]{burago01}). Let $\alg{f}(a)$ be the limit of the sequence
\begin{equation*}
  \left(\targetsetforward{s,j_0\circ j_1(n)}{a}{\Lip^{d(s)}(a),D}\right)_{n\in\N}
\end{equation*}
for $\Haus{\|\cdot\|_{{c(s)}}}$. By Lemma (\ref{targetset-forward-norm-lemma}), we have
\begin{equation*}
\lim_{n\rightarrow\infty} \diam{\targetsetforward{s,j_0\circ j_1 (n)}{a}{\Lip^{d(s)}(a),D}}{\|\cdot\|_{{c(s)}}} = 0
\end{equation*}
and thus $\alg{f}(a)$ is a singleton, as desired. We write $\alg{f}(a) = \{ f \}$.

In particular, we thus note that $f$ is the limit of any sequence $(b_n)_{n\in\N}$ with $b_n \in \targetsetforward{s,j_0\circ j_1(n)}{a}{\Lip^{d(s)}(a),D}$ for all $n\in\N$. Thus
\begin{equation*}
  \|f\|_{{c(s)}} = \lim_{n\rightarrow\infty} \|b_n\|_{{c(s)}} \leq D\|a\|_{{d(s)}}\text{.}
\end{equation*}

Moreover since $\Lip^{d(s)}$ is lower semicontinuous, we also have
\begin{equation*}
\Lip^{c(s)}(f) \leq \liminf_{n\rightarrow\infty} \Lip^{c(s)}(b_n) \leq D \Lip^{d(s)}(a) \text{.}
\end{equation*}

To conclude this step, let $l \geq \Lip^{d(s)}(a)$. Since for all $n\in\N$, we have
\begin{equation*}\label{main-thm-inclusion-eq}
\targetsetforward{s,n}{a}{\Lip^{d(s)}(a),D} \subseteq \targetsetforward{s,n}{a}{l,D}\text{, }
\end{equation*}
and since by Lemma (\ref{targetset-forward-norm-lemma}) we have
\begin{equation*}
\lim_{n\rightarrow\infty} \diam{\targetsetforward{s,n}{a}{l,D}}{\|\cdot\|_{{c(s)}_n}} = 0\text{,}
\end{equation*}
we conclude by Lemma (\ref{inclusion-lemma}) that
\begin{equation*}
\left(\targetsetforward{s,j_0\circ j_1(n)}{a}{l,D}\right)_{n\in\N}
\end{equation*}
converges to $\{ f \}$ as well.

\begin{step}
If $s\in\mathcal{S}$ and if $J : \N \rightarrow\N$ is a strictly increasing function, then there exists a strictly increasing function $j : \N \rightarrow \N$ and a function $\varphi^s : \alg{S}^{d(s)} \rightarrow \sa{\A^{c(s)}}$ such that:
\begin{equation*}
\forall a \in \alg{S}^{d(s)} \quad \|\varphi^s(a)\|_{{c(s)}} \leq D(s) \|a\|_{{d(s)}} \text{ and }\Lip^{c(s)} \circ\varphi^s(a) \leq D(s) \Lip^{d(s)}(a) \text{, }
\end{equation*}
and such that, for all $a \in \alg{S}^{d(s)}$ and $l\geq \Lip^{d(s)}(a)$, the sequence
\begin{equation*}
\left(\targetsetforward{s,J\circ j(n)}{a}{l,D(s)}\right)_{n\in\N}
\end{equation*}
converges for $\Haus{\|\cdot\|_{{c(s)}}}$ to $\{\varphi^s(a)\}$.
\end{step}

We now apply Step (1) repeatedly, using a diagonal argument. Fix $s\in\mathcal{S}$ and again, just set $D = D(s)$. Write $\alg{S}^{d(s)} = \left\{ a_n : n \in \N \right\}$. By Step (1), there exists $j_0 : \N\rightarrow\N$ strictly increasing and there exists an element $\sa{\\A^{c(s)}}$ which we denote as  $\varphi^s(a_0)$, such that for all $l\geq \Lip^{d(s)}(a_0)$, the sequence $\left(\targetsetforward{s,J\circ j_0(n)}{a_0}{l,D}\right)_{n\in\N}$ converges to $\{\varphi^s(a_0)\}$ for $\Haus{\|\cdot\|_{{c(s)}}}$, and $\|\varphi^s(a_0)\|_{{c(s)}} \leq D \|a_0\|_{{d(s)}}$ while $\Lip^{c(s)}\circ\varphi^s(a_0) \leq D \Lip^{d(s)}(a_0)$.

Assume now that for some $k\in\N$, there exist strictly increasing functions $j_0, \ldots, j_k$ from $\N$ to $\N$ and $\varphi^s(a_0),\ldots,\varphi^s(a_k) \in \sa{\A^{c(s)}}$ such that for all $m\in\{0,\ldots,k\}$ and for all $l\geq \Lip^{d(s)}(a_j)$, the sequence
\begin{equation*}
\left(\targetsetforward{s,J\circ j_0\circ\cdots\circ j_m(n)}{a_j}{l,D}\right)_{n\in\N}
\end{equation*}
converges to the singleton $\left\{\varphi^s(a_j)\right\}$ for $\Haus{\|\cdot\|_{{c(s)}}}$, with $\|\varphi(a_j)\|_{{c(s)}} \leq D \|a_j\|_{{d(s)}}$ and $\Lip^{c(s)}\circ\varphi^s(a_j)\leq D \Lip^{d(s)}(a_j)$.

By Step (1), there exists a strictly increasing function $j_{k+1} : \N \rightarrow \N$ such that, for all $l\geq \Lip^{d(s)}(a_{k+1})$, the sequence $\left( \targetsetforward{s,J\circ j_0\circ\cdots\circ j_{k+1}(n)}{a_{k+1}}{l,D} \right)_{n\in\N}$ converges, for $\Haus{\|\cdot\|_{{c(s)}}}$, to some singleton denoted by $\{\varphi^s(a_{k+1})\}$ with $\|\varphi^s(a_{k+1})\|_{{c(s)}} \leq D \|a_{k+1}\|_{{d(s)}}$ and $\Lip^{c(s)}\circ\varphi^s(a_{k+1}) \leq D \Lip^{d(s)}(a_{k+1})$. This completes our induction.

Now, let $j : m \in \N \longmapsto J\circ j_0 \circ\cdots\circ j_m(m)$. By construction, for all $k\in\N$ and for all $l\geq\Lip^{d(s)}(a_k)$, the sequence $\left(\targetsetforward{s,j(n)}{a_k}{l,D}\right)_{n\in\N}$ converges to $\{\varphi^s(a_k)\}$ for $\Haus{\|\cdot\|_{{c(s)}}}$.

\begin{step}
There exists a strictly increasing function $j : \N \rightarrow \N$ such that, for all $s\in\mathcal{S}$ and for all  $a\in\alg{S}^{d(s)}$, there exists $\varphi^s(a) \in \sa{\A^{c(s)}}$, with
\begin{equation*}
\|\varphi^s(a)\|_{{c(s)}} \leq D(s) \|a\|_{{d(s)}} \text{ and }\Lip^{c(s)} \circ\varphi^s(a) \leq D(s) \Lip^{d(s)}(a) \text{, }
\end{equation*}
such that for all $l \geq \Lip^{d(s)}(a)$, the sequence
\begin{equation*}
\left(\targetsetforward{s,j(n)}{a}{l,D(s)}\right)_{n\in\N}
\end{equation*}
converges for $\Haus{\|\cdot\|_{{c(s)}}}$ to $\{\varphi^s(a)\}$.
\end{step}

Since $\mathcal{S}$ is countable, we write it as $\mathcal{S} = \{ s_m : m \in \N \}$ for this step. We now apply Step (2) repeatedly. First, let $j_0 : \N \rightarrow \N$ and $\varphi^{s_0}$ be given by Step (2) for $s_0$ and $J : n \in \N \mapsto n$.

Assume we have constructed, for some $k\in\N$, strictly increasing maps $j_0$,\ldots,$j_k$ from $\N$ to $\N$, and functions $\varphi^{s_0} : \alg{S}^{d(s_0)}\rightarrow\sa{\A^{c(s_0)}}$,\ldots,$\varphi^{s_k} : \alg{S}^{d(s_k)}\rightarrow\sa{\A^{c(s_k)}}$ such that for all $m\in\{0,\ldots,k\}$, $a\in\alg{S}^{d(s_m)}$, and $l\geq\Lip^{d(s_m)}(a)$, we have
\begin{equation*}
\Lip^{c(s_m)}\circ\varphi^{s_m}(a) \leq D(s_m) \Lip^{d(s)}(a) \text{ and }\|\varphi^{s_m}(a)\|_{{c(m)}} \leq D(s_m) \|a\|_{{d(s_m)}} \text{,}
\end{equation*}
and the sequence $\left(\targetsetforward{s_m, j_0\circ\ldots\circ j_m(n)}{a}{l,D(s_m)}\right)_{n\in\N}$ converges to $\{\varphi^{s_m}(a)\}$ for $\Haus{\|\cdot\|_{{c(s)}}}$.

We apply Step (2) with $J : n \in \N \rightarrow j_0\circ\ldots\circ j_k$ and $s_{k+1}$: thus there exists $j_{k+1} : \N \rightarrow \N$ strictly increasing and $\varphi^{s_{k+1}} : \alg{S}^{d(k+1)} \rightarrow \sa{\A^{c(k+1)}}$ satisfying the conclusions of Step (2). This completes our induction.

We conclude this step by setting $j : m \in \N \longmapsto j_0 \circ\cdots \circ j_m(m)$.

\begin{step}
There exists a strictly increasing function $j : \N \rightarrow \N$ and, for all $s\in\mathcal{S}$, a function $\varphi^s : \dom{\Lip^{d(s)}} \rightarrow \sa{\A^{c(s)}}$ such that for all $a\in\dom{\Lip^{d(s)}}$, for all $l\geq \Lip^{d(s)}(a)$, and for all $D \geq D(s)$, the sequence
\begin{equation*}
\left( \targetsetforward{s,j(n)}{a}{l,D} \right)_{n\in\N}
\end{equation*}
converges to $\{\varphi^s(a)\}$ for $\Haus{\|\cdot\|_{{c(s)}}}$; moreover
\begin{equation*}
\|\varphi^s(a)\|_{{c(s)}} \leq D(s)\|a\|_{{d(s)}}\text{ and }\Lip^{c(s)}(\varphi^s(a)) \leq D(s) \Lip^{d(s)}(a)\text{.}
\end{equation*}
\end{step}

Let $j$ and, for all $s\in\mathcal{S}$, let $\varphi^s : \mathfrak{S}^{d(s)} \rightarrow\sa{\A^{c(s)}}$ be constructed by Step (3). Fix $s\in\mathcal{S}$ and write $D \geq D(s)$.

Let $a\in\dom{\Lip^{d(s)}}$, and $\varepsilon > 0$. There exists $a'\in\alg{S}^{d(s)}$ such that $\|a-a'\|_{{d(s)}} < \frac{\varepsilon}{4 D}$. Let $l = \max\{\Lip^{d(s)}(a),\Lip^{d(s)}(a')\}$. Since by Lemma (\ref{targetset-forward-norm-lemma})
\begin{equation*}
\lim_{n\rightarrow\infty} \diam{\targetsetforward{s,n}{a'}{l,D}}{\|\cdot\|_{{c(s)}}} = 0
\end{equation*}
and since $\left(\targetsetforward{s,j(n)}{a'}{l,D(s)}\right)_{n\in\N}$ converges to a singleton for $\Haus{\|\cdot\|_{{c(s)}}}$, by Lemma (\ref{inclusion-lemma}), the sequence $\left(\targetsetforward{s,j(n)}{a'}{l,D}\right)_{n\in\N}$ also converges to the same singleton.

Since $\left(\targetsetforward{s,j(n)}{a'}{l,D}\right)_{n\in\N}$ is convergent for the Hausdorff distance $\Haus{\|\cdot\|_{{c(s)}}}$ by Step (3), it is Cauchy for $\Haus{\|\cdot\|_{{c(s)}}}$. Let $N\in \N$ such that, for all $p,q \geq N$, we have
\begin{equation*}
\Haus{\|\cdot\|_{{c(s)}}}\left(\targetsetforward{s,j(p)}{a'}{l,D}, \targetsetforward{s,j(q)}{a'}{l,D}\right) < \frac{\varepsilon}{3} \text{.}
\end{equation*}

By Lemma (\ref{targetset-forward-norm-lemma}), we have, for all $l \geq \max\{\Lip^{d(s)}(a),\Lip^{d(s)}(a')\}$, for all $n\in\N$, for all $b \in \targetsetforward{s,j(n)}{a}{l,D}$ and for all $b'\in\targetsetforward{s,j(n)}{a'}{l,D}$

\begin{equation*}
\left\| b - b' \right\|_{{c(s)}} \leq D\left(\|a - a'\|_{{d(s)}} + 2 l \left(\frac{1}{j(n)+1} + \frac{1}{j(n)+1}\right)\right) \leq \frac{\varepsilon}{4} + \frac{4 l}{j(n)+1}\text{.}
\end{equation*}

Let $M \in \N$ such that for all $n\geq M$, we have $\frac{4 l}{j(n)+1} \leq \frac{\varepsilon}{12}$, so that
\begin{equation*}
\Haus{\|\cdot\|_{{c(s)}}}\left(\targetsetforward{s,j(n)}{a}{l,D}, \targetsetforward{s,j(n)}{a'}{l,D}\right) \leq \frac{\varepsilon}{3}\text{.}
\end{equation*}

Hence for all $p,q \geq \max\{N,M\}$
\begin{multline*}
\Haus{\|\cdot\|_{{c(s)}}}\left(\targetsetforward{s,j(p)}{a}{l,D},\targetsetforward{s,j(q)}{a}{l,D}\right) \\
\begin{split}
& \leq \Haus{\|\cdot\|_{{c(s)}}}\left(\targetsetforward{s,j(p)}{a}{l,D},\targetsetforward{s,j(p)}{a'}{l,D}\right) \\
&\quad + \Haus{\|\cdot\|_{{c(s)}}}\left(\targetsetforward{s,j(p)}{a'}{l,D},\targetsetforward{s,j(q)}{a'}{l,D}\right) \\
&\quad + \Haus{\|\cdot\|_{{c(s)}}}\left(\targetsetforward{s,j(q)}{a'}{l,D},\targetsetforward{j(q)}{a}{l,D}\right)\\
&< \varepsilon \text{.}
\end{split}
\end{multline*}

Hence $\left(\targetsetforward{s,j(n)}{a}{l,D}\right)_{n\in\N}$ is a Cauchy sequence for $\Haus{\|\cdot\|_{{c(s)}}}$. Since $(\sa{\A^{c(s)}},\|\cdot\|_{{c(s)}})$ is complete, so is $\Haus{\|\cdot\|_{{c(s)}}}$ and thus $\left(\targetsetforward{s,j(n)}{a}{l,D}\right)_{n\in\N}$ converges for $\Haus{\|\cdot\|_{{c(s)}}}$ to some set $\alg{f}(a)$; again since $\lim_{n\rightarrow\infty} \diam{\targetsetforward{s,j(n)}{a}{l,D}}{\|\cdot\|_{{c(s)}}} = 0$, the set $\alg{f}(a)$ is a singleton which we denote by $\{\varphi^s(a)\}$. We observe that if in fact $a\in\alg{S}^{d(s)}$, then we introduce no confusion with this notation since by construction, $\left(\targetsetforward{s,j(n)}{a}{l,D}\right)_{n\in\N}$ converges to $\{\varphi^s(a)\}$ as defined by Step (3).

Using Lemma (\ref{inclusion-lemma}) and the observation that for all $n\in\N$, and for all $l'' \geq l' \geq \Lip^{d(s)}(a)$, we have
\begin{equation*}
\targetsetforward{s,j(n)}{a}{l',D} \subseteq \targetsetforward{s,j(n)}{a}{l'',D}
\end{equation*}
we conclude that for all $l\geq \Lip^{d(s)}(a)$, the sequence
\begin{equation*}
\left(\targetsetforward{s,j(n)}{a}{l,D}\right)_{n\in\N}
\end{equation*}
converges for $\Haus{\|\cdot\|_{{c(s)}}}$ to $\{\varphi^s(a)\}$.

In particular, if $a\in\dom{\Lip^{d(s)}}$, then $\left(\targetsetforward{s,j(n)}{a}{\Lip^{d(s)}(a),D(s)}\right)_{n\in\N}$ converges to $\{\varphi^s(a)\}$, and thus as before, since $\Lip^{c(s)}$ is lower semi-continuous, we conclude that $\Lip^{c(s)}\circ\varphi^s(a) \leq D(s) \Lip^{d(s)}(a)$ for all $a\in\dom{\Lip^{d(s)}}$; Similarly $\|\varphi(a)\|_{{c(s)}} \leq D(s) \|a\|_{{d(s)}}$.

\begin{step}
For all $s\in\mathcal{S}$, the map $\varphi^s : \dom{\Lip^{d(s)}} \rightarrow \dom{\Lip^{c(s)}}$ is linear and $\varphi^s(\unit_{\A^{d(s)}}) = \unit_{\A^{c(s)}}$.
\end{step}

Fix $s\in\mathcal{S}$ and write $D = D(s)$. Let $a,a' \in \dom{\Lip^{d(s)}}$, $t \in \R$, and $l \geq \max\{\Lip^{d(s)}(a),\Lip^{d(s)}(a')\}$. For each $n\in\N$, let $b_n \in \targetsetforward{s,j(n)}{a}{l,D}$ and $b_n'\in\targetsetforward{s,j(n)}{a'}{l,D}$. By construction, $(b_n)_{n\in\N}$ converges to $\varphi^s(a)$ while $(b'_n)_{n\in\N}$ converges to $\varphi^s(a')$ in $\sa{\A^{c(s)}}$. On the other hand, by Lemma (\ref{targetset-forward-linear-morphism-lemma})
\begin{equation*}
\forall n \in \N \quad b_n + t b_n' \in \targetsetforward{s,j(n)}{a + t a'}{(1+|t|)l,D(s)}\text{,}
\end{equation*}
so $(b_n + t b_n')_{n\in\N}$ converges to $\varphi^s(a + t a')$ in $\sa{\A^{c(s)}}$. By uniqueness of the limit, we conclude
\begin{equation*}
\varphi^s(a + t a') = \varphi^s(a) + t \varphi^s(a') \text{.}
\end{equation*}

We also note that since $\varphi_n^s$ maps unit to unit, we have $\unit_{\A^{c(s)}} \in \targetsetforward{s,n}{\unit_{\A^{d(s)}}}{0,D}$ for all $n\in\N$. Thus $\varphi^{s} (\unit_{d(s)}) = \unit_{\A_{c(s)}}$.

\begin{step}
For all $s\in\mathcal{S}$, the map $\varphi^s$ has a unique extension (still denoted by $\varphi^s$) as a continuous linear map from $\A^{d(s)}$ to $\A^{c(s)}$ with:
\begin{itemize}
\item $\forall a \in \sa{\A^{d(s)}} \quad \norm{\varphi^s(a)}{c(s)} \leq D(s) \norm{a}{d(s)}$,
\item $\opnorm{\varphi^s}{s}{} \leq \sqrt{2}D(s)$,
\item $\Lip^{c(s)}\circ\varphi^s \leq D(s) \Lip^{d(s)} $.
\item $\forall a\in\A^{d(s)} \quad \varphi^s(a^\ast) = \varphi^s(a)^\ast$.
\end{itemize}
\end{step}

Fix $s\in\mathcal{S}$. The map $\varphi^s : \dom{\Lip^{d(s)}} \rightarrow\dom{\Lip^{c(s)}}$ satisfies $\Lip^{c(s)}\circ\varphi^s\leq D(s) \Lip^{d(s)}$ by Step (4).

For all $a\in\dom{\Lip^{d(s)}}$, we have $\|\varphi^s(a)\|_{{c(s)}} \leq D(s) \|a\|_{{d(s)}}$ by construction, so $\varphi^s : \dom{\Lip^{d(s)}} \rightarrow \sa{\A^{c(s)}}$ is a continuous linear map. Hence it is uniformly continuous on the dense subspace $\dom{\Lip^{d(s)}}$ of $\sa{\A^{d(s)}}$ and therefore, it has a unique uniformly continuous extension as a linear map from $\sa{\A^{d(s)}}$ to $\sa{\A^{c(s)}}$ with norm at most $D(s)$. As a side note, if $\Lip^{d(s)}(a) = \infty$ (equivalently if $a\in\sa{\A^{d(s)}}\setminus\dom{\Lip^{d(s)}}$) then the inequality $\Lip^{c(s)}\circ\varphi^s(a) \leq D(s) \Lip^{d(s)}(a) = \infty$ is trivial.

For all $a\in \A^{d(s)}$, we then set
\begin{equation*}
\varphi^s(a) = \varphi^s\left(\Re a\right) + i \varphi^s\left(\Im a\right)\text{,}
\end{equation*}
where $\Re a = \frac{a + a^\ast}{2}$ and $\Im a=  \frac{a - a^\ast}{2}$ are,  respectively, the real and imaginary parts of $a$.

First, we note that if $a\in\sa{\A^{d(s)}}$ then $\Im a = 0$ and thus our extension of $\varphi^s$ to $\A^{d(s)}$ agrees with $\varphi^s$ on $\sa{\A^{d(s)}}$, which justifies that we keep the notation $\varphi^s$. Moreover, it is straightforward that $\varphi^s$ is linear --- owing to the linearity of $\varphi^s$ restricted to $\sa{\A^{d(s)}}$.

A quick computation shows that for all $a\in A$
\begin{align*}
  \left\|\varphi^s(a)\right\|_{{c(s)}}^2 
  &= \left\|\varphi^s(a)^\ast \varphi^s(a)\right\|_{{c(s)}} \\
  &\leq \|\varphi^s(\Re a)^2 + \varphi^s(\Im a)^2\|_{{c(s)}}\\
  &\leq 2 D(s)^2 \|a\|_{{c(s)}}^2 \text{.}
\end{align*}

Last, by construction, $\varphi^s(a^\ast) = \varphi^s(\Re a - i \Im a) = \varphi^s(\Re a) - i \varphi^s(\Im a) = \varphi^s(a)^\ast$ so $\varphi^s$ preserves the adjoint operation.

This concludes our step.

\begin{step}
For all $s,s'\in\mathcal{S}$, if $d(s)=c(s')$, then $\varphi^{s} \circ \varphi^{s'} = \varphi^{s\circ s'}$.
\end{step}

Let $s,s'\in\mathcal{S}$ such that $d(s) = c(s')$. By assumption, $(\A^{d(s)},\Lip^{d(s)}) = (\A^{c(s')},\Lip^{c(s')})$, so the statement of this step is at least meaningful. We also note that it is obviously sufficient to prove that $\varphi^s\circ\varphi^{s'}(a)= \varphi^{s\circ s'}(a)$ for all $a\in\sa{\A^{d(s)}}$.

Moreover, by assumption
\begin{equation*}
\forall n \in \N \quad (\A^{c(s')}_n,\Lip_n^{c(s')}) = \codom{\varphi_n^{s'}} = \dom{\varphi_n^s} = (\A^{d(s)}_n,\Lip_n^{d(s)}) \text{.}
\end{equation*}

Let $a\in\sa{\A^{d(s')}}$ and write $l = \Lip^{d(s')}(a)$. For all $n\in\N$, let $a_n \in \targetsettunnel{\tau_{j(n)}^{d(s')}}{a}{l}$, and write $b_n = \varphi_{j(n)}^{s'}(a_n)$. Note that $b_n \in \A_{j(n)}^{c(s')} = \A_{j(n)}^{d(s)}$. 

Let $c_n \in \targetsettunnel{\gamma_{j(n)}^{c(s')}}{b_n}{D(s') l}$ --- in particular, $c_n \in \targetsetforward{s',j(n)}{a}{l, D(s')} \subseteq \A^{c(s')}$.

In addition, let $e_n \in \targetsettunnel{\gamma_{j(n)}^{c(s)}}{\varphi_{j(n)}^{s\circ s'}(a_n)}{D(s\circ s')l}$, so that by Definition (\ref{targetset-forward-def}), we have $e_n \in \targetsetforward{s\circ s',j(n)}{a}{l, D(s\circ s')}$ (so in particular, $e_n \in \A^{c(s)}$). By construction, $(e_n)_{n\in\N}$ converges to $\varphi^{s\circ s'}(a)$ in $\A^{c(s)}$, while $(c_n)_{n\in\N}$ converges to $\varphi^{s'}(a)$ in $\A^{d(s)}$.

Let $d_n \in \targetsettunnel{\gamma_{j(n)}^{c(s)}}{\varphi_{j(n)}^s(b_n)}{D(s) D(s') l}$ --- so $d_n \in \A^{c(s)}$. Now, since we chose $c_n \in \targetsettunnel{\gamma_{j(n)}^{d(s)}}{b_n}{D(s') l}$, and since $\gamma_{j(n)}^{d(s)} = \left(\tau_{j(n)}^{d(s)}\right)^{-1}$, we also have by symmetry $b_n \in \targetsettunnel{\tau_{j(n)}^{d(s)}}{c_n}{D(s') l}$. So by Definition (\ref{targetset-forward-def}), we conclude that $d_n \in \targetsetforward{s,j(n)}{c_n}{D(s')l,D(s)}$.

Let $h_n \in \targetsetforward{s,j(n)}{\varphi^{s'}(a)}{D(s')l,D(s)}$. By construction, $(h_n)_{n\in\N}$ converges to $\varphi^s\circ \varphi^{s'}(a)$ in $\A^{c(s)}$. On the other hand, by Lemma (\ref{targetset-forward-norm-lemma})
\begin{equation*}
\|d_n - h_n\|_{{c(s)}} \leq D(s) \left(\|c_n - \varphi^{s'}(a)\|_{{d(s)}} + 4 l D(s') \tunnelextent{\tau^s_{j(n)}}\right) 
\end{equation*}
so
\begin{multline*}
0\leq \limsup_{n\rightarrow\infty} \|d_n - h_n\|_{{c(s)}} \\ \leq D(s) \left( \lim_{n\rightarrow\infty} \|c_n - \varphi^{s'}(a)\|_{{c(s')}} + 4 l D(s') \lim_{n\rightarrow\infty} \tunnelextent{\tau^s_{j(n)}} \right) = 0
\end{multline*}
so $(d_n)_{n\in\N}$ converges to $\varphi^s\circ\varphi^{s'}(a)$ as well.

Let $\varepsilon > 0$ and let $D = \max\{D(s)D(s'), D(s\circ s')\}$.

Let $N_1\in\N$ such that for all $n\geq N_1$, we have $\|d_n - \varphi^s\circ\varphi^{s'}(a)\|_{{c(s)}} < \frac{\varepsilon}{4}$.

Let $N_2\in\N$ such that for all $n\geq N_2$, we have $\|e_n - \varphi^{s\circ s'}(a)\|_{{c(s)}} < \frac{\varepsilon}{4}$.

Let $N_3\in\N$ such that for all $n\geq N_3$, we have $\tunnelextent{\tau_{j(n)}^s} < \frac{\varepsilon}{8 D l}$.

Last, by assumption, there exists $N_4 \in \N$ such that for all $n\geq N_4$ and for all $a' \in \sa{\A^{d(s')}}$ with $\Lip^{d(s')}(a') \leq l$, we have
\begin{equation*}
\left\|\varphi_{j(n)}^{s\circ s'}(a') - \varphi_{j(n)}^s\circ\varphi_{j(n)}^{s'}(a')\right\|_{{c(s),j(n)}} < \frac{\varepsilon}{4} \text{.}
\end{equation*}

We record that by assumption, we have
\begin{equation*}
d_n \in \targetsettunnel{\gamma_{j(n)}^{c(s)}}{\varphi_{j(n)}^{s}(b_n)}{D(s)D(s')l}\subseteq \targetsettunnel{\gamma_{j(n)}^{c(s)}}{\varphi_{j(n)}^{s}(b_n)}{Dl}
\end{equation*}
and
\begin{equation*}
e_n \in \targetsettunnel{\gamma_{j(n)}^{c(s)}}{\varphi_{j(n)}^{s\circ s'}(a)}{D(s\circ s') l}\subseteq\targetsettunnel{\gamma_{j(n)}^{c(s)}}{\varphi_{j(n)}^{s\circ s'}(a)}{Dl}\text{.}
\end{equation*}

For all $n\geq \max\{N_1,N_2,N_3,N_4\}$, we thus have
\begin{multline*}
0\leq \|\varphi^{s\circ s'}(a) - \varphi^s\circ\varphi^{s'}(a)\|_{{c(s)}} \\
\begin{split}
&\leq \|\varphi^{s\circ s'}(a) - e_n\|_{{c(s)}} + \|e_n - d_n\|_{{c(s)}} + \|d_n - \varphi^{s}\circ\varphi^{s'}(a)\|_{{c(s)}} \\
&\leq \frac{2\varepsilon}{4} + \|e_n - d_n\|_{{c(s)}} \\
&\leq \frac{2\varepsilon}{4} + \|\varphi_{j(n)}^{s\circ s'} (a_n) - \varphi_{j(n)}^s(b_n) \|_{{{c(s),j(n)}}} + 2 l D \tunnelextent{\tau_{j(n)}^{c(s)}} \text{ by Prop (\ref{targetset-norm-prop})} \\
&\leq \frac{3\varepsilon}{4} + \|\varphi_{j(n)}^{s\circ s'} (a_n) - \varphi_{j(n)}^s(\varphi_{j(n)}^{s'}(a_n)) \|_{{c(s),j(n)}} \\
&\leq \varepsilon \text{.}
\end{split}
\end{multline*}

Hence, as $\varepsilon > 0$ is arbitrary, we conclude
\begin{equation*}
\left\|\varphi^{s\circ s'}(a) - \varphi^s\circ\varphi^{s'}(a)\right\|_{{c(s)}} = 0 \text{ so }\varphi^{s\circ s'}(a) = \varphi^s\circ\varphi^{s'}(a) \text{.}
\end{equation*}

This concludes this step.


We now prove a non-triviality result --- which also completes the proof that we have constructed a functor.
\begin{step}
If, for some $s,s'\in \mathcal{S}$, we have $c(s) = c(s')$ and $d(s) = d(s')$, then
\begin{equation*}
\liminf_{n\rightarrow\infty} \KantorovichDist{\Lip^{d(s)}_n}{\A^{c(s)}_n}{\varphi_n^s}{\varphi_n^{s'}} \leq \KantorovichDist{\Lip^{d(s)}}{\A^{c(s)}}{\varphi^s}{\varphi^{s'}} \leq  \limsup_{n\rightarrow\infty} \KantorovichDist{\Lip_n^{d(s)}}{\A^{c(s)}_n}{\varphi_n^s}{\varphi_n^{s'}}\text{.}
\end{equation*}

Moreover, if for some $s\in \mathcal{S}$, we have $d(s) = c(s)$, then
\begin{equation*}
\liminf_{n\rightarrow\infty} \KantorovichLength{\Lip_n}(\varphi_n^s) \leq \KantorovichLength{\Lip}(\varphi^s) \leq \limsup_{n\rightarrow\infty} \KantorovichLength{\Lip_n}(\varphi_n^s)\text{.}
\end{equation*}
\end{step}

Fix $s,s' \in \mathcal{S}$ with the desired properties and write $D = \max\{1,D(s),D(s')\}$. Let $\varepsilon > 0$. Let $E$ be a finite subset of $\{a \in \sa{\A^{d(s)}} : \Lip^{d(s)}(a) \leq 1\}$ such that for all $a\in \sa{\A^{d(s)}}$ with $\Lip^{d(s)}(a) \leq 1$ there exists $a'\in E$ and $t\in\R$ such that
\begin{equation*}
\|a - (a'+t\unit_\A)\|_{{d(s)}} < \frac{\varepsilon}{12 D} \text{.}
\end{equation*}

Let $N \in \N$ such that for all $n\geq N$, we have for all $a'\in E$
\begin{equation*}
\Haus{\|\cdot\|_{{d(s)}}}\left(\targetsetforward{s,j(n)}{a'}{1,D}, \{\varphi^s(a')\}\right) <  \frac{\varepsilon}{8}
\end{equation*}
and
\begin{equation*}
\Haus{\|\cdot\|_{{d(s)}}}\left(\targetsetforward{s',j(n)}{a'}{1,D}, \{\varphi^{s'}(a')\}\right) <  \frac{\varepsilon}{8}
\end{equation*}
while
\begin{equation*}
\tunnelextent{\tau_{j(n)}^s} < \frac{\varepsilon}{12 D}\text{.}
\end{equation*}
Note that $N$ exists since $E$ is finite and by construction of $\varphi^s$ above.

Fix $n\geq N$.

Let $G_n\subseteq \A_{j(n)}^{d(s)}$ be given by Lemma (\ref{totally-bounded-correspondance-lemma}), so that
\begin{itemize}
\item for all $a\in\sa{\A_{j(n)}^{d(s)}}$ with $\Lip^{d(s)}_n(a)\leq 1$, there exists $t \in \R$ and $a' \in G_n$ such that $\|a - (a' + t\unit_\A)\|_{{d(s)},n} \leq \frac{\varepsilon}{4 D }$,
\item for all $a\in G_n$ there exists $a'\in E$ such that $a \in \targetsettunnel{\tau_{j(n)}^{d(s)}}{a'}{1}$, and conversely, for all $a\in E$ there exists $a'\in G_n$ such that $a \in \targetsettunnel{\gamma_{j(n)}^{d(s)}}{a'}{1}$.
\end{itemize}

Let $a\in\sa{\A^{d(s)}}$ with $\Lip^{d(s)}(a) \leq 1$. There exists $a' \in E$ and $t\in\R$ such that $\|a  - (a' + t\unit_\A)\|_{{d(s)}} < \frac{\varepsilon}{12 D} < \frac{\varepsilon}{4 D}$. We then have
\begin{equation*}
  \begin{split}
    \left\| \varphi^{s'}(a) - \varphi^s(a) \right\|_{{c(s)}} &\leq \left\|\varphi^{s'}\left( a - (a' + t\unit_\A) \right) \right\|_{{c(s)}} \\
    &\quad + \left\|\varphi^{s'}(a' + t\unit_\A) - \varphi^s(a' + t\unit_\A)\right\|_{{c(s)}} \\
    &\quad + \left\|\varphi^s(a - (a' + t\unit_\A))\right\|_{{c(s)}} \\
    &\leq \frac{D \varepsilon}{4 D} + \|\varphi^{s'}(a') - \varphi^s(a')\|_{{c(s)}} + \frac{D \varepsilon}{4 D} \\
    &\leq \frac{\varepsilon}{2} + \|\varphi^{s'}(a') -\varphi^s(a')\|_{{c(s)}} \text{.}
  \end{split}
\end{equation*}

Let $a_n \in G_n \cap \targetsettunnel{\tau_{j(n)}^s}{a'}{1}$. Let $c_n = \varphi_{j(n)}^s(a_n)$ and $d_n = \varphi_{j(n)}^{s'}(a_n)$. Let $c \in \targetsettunnel{\gamma_{j(n)}}{c_n}{D}$ and $d\in\targetsettunnel{\gamma_{j(n)}}{d_n}{D}$. In particular, $c \in \targetsetforward{s,j(n)}{a'}{1,D}$ and $d\in\targetsetforward{s',j(n)}{a'}{1,D}$. Thus by construction, $\|\varphi^s(a') - c\|_{{c(s)}} \leq \frac{\varepsilon}{8}$ and $\|\varphi^{s'}(a')-d\|_{c(s)} \leq \frac{\varepsilon}{8}$.

We then compute
\begin{multline*}
\|\varphi^{s'}(a') - \varphi^s(a')\|_{{c(s)}}\\
\begin{split}
&\leq \frac{\varepsilon}{4} + \|d - c\|_{{c(s)}} \\
&\leq \frac{\varepsilon}{4} + 2\frac{\varepsilon}{12 D} + \|\varphi_{j(n)}^{s'}(a_n) - \varphi_{j(n)}^s(a_n)\|_{{c(s),j(n)}} \text{ by Proposition (\ref{targetset-norm-prop})} \\
&\leq \frac{\varepsilon}{2} + \KantorovichDist{s}{j(n)}{\varphi_{j(n)}^s}{\varphi_{j(n)}^{s'}}  \text{.}
\end{split}
\end{multline*}

Hence for all $a\in\sa{\A}$ with $\Lip^{d(s)}(a) \leq 1$
\begin{equation*}
\| \varphi^{s}(a) - \varphi^{s'}(a) \|_{{c(s)}} \leq \frac{\varepsilon}{2} + \frac{\varepsilon}{2} + \KantorovichDist{s}{j(n)}{\varphi_{j(n)}^s}{\varphi^{s'}_{j(n)}} \text{,}
\end{equation*}
i.e. $\KantorovichDist{s}{}{\varphi^s}{\varphi^{s'}} \leq \KantorovichDist{s}{j(n)}{\varphi_{j(n)}^s}{\varphi_{j(n)}^{s'}} + \varepsilon \text{.}$

We proceed symmetrically to show the converse inequality. Namely, let $a_n \in \A^{d(s)}_{j(n)}$ with $\Lip^{d(s)}_{\A_{j(n)}}(a_n) \leq 1$. There exists $t \in \R$ and $a_n' \in G_n$ such that $\|a_n - (a_n' + t\unit_\A)\|_{{d(s),j(n)}} < \frac{\varepsilon}{4 D}$. As before, we have
\begin{equation*}
\|\varphi_{j(n)}^{s'}(a_n) - \varphi_{j(n)}^{s}(a_n) \|_{{c(s)},{j(n)}} \leq \frac{\varepsilon}{2} + \|\varphi_{j(n)}^{s'}(a_n') - \varphi^s_{j(n)}(a_n')\|_{{c(s)},{j(n)}} \text{.}
\end{equation*}

Let $a\in E$ such that $a \in \targetsettunnel{\gamma^{d(s)}_{j(n)}}{a_n'}{1}$. Let $c \in \targetsettunnel{\gamma_{j(n)}^{d(s)}}{\varphi^s_{j(n)}(a_n')}{D}$. Note that by symmetry, we have $a_n \in\targetsettunnel{\tau^{d(s)}_{j(n)}}{a}{1}$, so that by construction $c \in \targetsetforward{s,j(n)}{a}{1,D}$ and thus $\|\varphi^s(a) - c\|_{{c(s)}} < \frac{\varepsilon}{8}$.

Similarly, let $d\in\targetsettunnel{\gamma_{j(n)}^{d(s)}}{\varphi^{s'}_{j(n)}(a_n')}{D}$, and note $d\in\targetsetforward{s',j(n)}{a}{1,D}$ so $\|\varphi^{s'}(a) - d\|_{c(s)} < \frac{\varepsilon}{8}$.

From this, we then obtain as before that $\|\varphi_{j(n)}^{s'}(a_n') - \varphi_{j(n)}^s(a_n')\|_{{c(s)},{j(n)}} \leq \frac{\varepsilon}{2} + \KantorovichDist{s}{}{\varphi^{s'}}{\varphi^s}$. Thus we get
\begin{equation*}
\KantorovichDist{s}{j(n)}{\varphi^s_{j(n)}}{\varphi_{j(n)}^{s'}} \leq \KantorovichDist{s}{}{\varphi^s}{\varphi^{s'}} + \varepsilon \text{.}
\end{equation*}

Hence we have shown that, for all $\varepsilon > 0$, there exists $N\in\N$ such that for all $n\geq N$, we have
\begin{equation*}
\KantorovichDist{s}{}{\varphi^s}{\varphi^{s'}} - \varepsilon \leq \KantorovichDist{s}{{j(n)}}{\varphi^s_{j(n)}}{\varphi_{j(n)}^{s'}} \leq \KantorovichDist{s}{}{\varphi^s}{\varphi^{s'}} + \varepsilon\text{.}
\end{equation*}

Thus
\begin{equation*}
\lim_{n\rightarrow\infty} \KantorovichDist{s}{j(n)}{\varphi^s_{j(n)}}{\varphi^{s'}_{j(n)}} = \KantorovichDist{s}{}{\varphi^s}{\varphi^{s'}}\text{.}
\end{equation*}
In particular, we have
\begin{equation*}
\liminf_{n\rightarrow\infty} \KantorovichDist{s}{n}{\varphi^s_n}{\varphi^{s'}_n} \leq \KantorovichDist{s}{}{\varphi^s}{\varphi^{s'}} \leq \limsup_{n\rightarrow\infty} \KantorovichDist{s}{n}{\varphi^s_n}{\varphi^{s'}_n} \text{.}
\end{equation*}

If $c(s)=d(s)$ then our argument above can be applied as well when we replace $\varphi_n^{s'}$ and $\varphi^{s'}$, respectively, with the identity of $\mathrm{id}_n$ of $\A_{j(n)}^{d(s)}$ and the identity $\mathrm{id}$ of $\A^{d(s)}$, and adjusting the proof accordingly, in which case we conclude
\begin{equation*}
\liminf_{n\rightarrow\infty} \KantorovichDist{s}{n}{\varphi_n^s}{\mathrm{id}_n} \leq \KantorovichDist{s}{}{\varphi^s}{\mathrm{id}} \leq \limsup_{n\rightarrow\infty} \KantorovichDist{s}{n}{\varphi_n^s}{\mathrm{id}_n} \text{.}
\end{equation*}

In particular, if $s \in \mathcal{S}$ is a unit, by assumption and our previous computation, we then conclude $\KantorovichDist{s}{}{\varphi^s}{\mathrm{id}} = 0$, so $\varphi^s$ is the identity of $\A^{d(s)}$.


\begin{summary}
By setting $\Psi(s) = (\A^s,\Lip^s)$ for all $s\in\mathcal{S}^{(0)}$, and $\Psi(s) = \varphi^s$ for all $s\in \mathcal{S}$, we have proven that $\Psi$ is a functor from $(\mathcal{S},\mathcal{S}^{(0)},d,c,\circ)$ to the category of {\Qqcms{F}s} whose arrows are Lipschitz linear maps. As any functor, $\Psi$ preserves invertibility and it maps inverses to inverses as well.
\end{summary}

We now show that under appropriate assumptions, the functor we just constructed is in fact valued in the category of compact quantum metric spaces with arrows as positive Lipschitz linear maps or as Lipschitz morphisms (the latter being the default category for compact quantum metric spaces).

\begin{step}
If, for some $s\in\mathcal{S}$ and for all $n\in\N$, the map $\varphi_n^s$ is positive, and $\{a\in\dom{\Lip^{d(s)}}:a\geq 0\}$ is dense in $\{a\in\sa{\A^{d(s)}}: a \geq 0\}$, then $\varphi^s$ is positive as well.
\end{step}

Let $a\in\dom{\Lip^{d(s)}}$ with $a\geq 0$ and write $l = \Lip^{d(s)}(a)$; we also write $D = D(s)$. Let $\mu \in \StateSpace\left(\A^{c(s)}\right)$. Let $\varepsilon > 0$. There exists $N_1 \in \N$ such that, for all $n\geq N_1$, we can find $f_n \in \A^{c(s)}$ such that $\|f_n - \varphi^s(a)\|_{\A^{c(s)}} < \varepsilon$, with $f_n \in \targetsettunnel{\gamma_{j(n)}^{c(s)}}{\varphi_n^s(b_n)}{D l}$, where $b_n \in \targetsettunnel{\tau_{j(n)}^{d(s)}}{a}{l}$. 

There exists $N_2 \in \N$ such that for all $n\geq N_2$ we have
\begin{equation*}
\tunnelextent{\tau_{j(n)}^{d(s)}} < \frac{\varepsilon}{3 (l + 1)} \text{ and }\tunnelextent{\gamma_{j(n)}^{c(s)}} < \frac{\varepsilon}{3 (D l + 1)} \text{.}
\end{equation*}

Let $N = \max\{N_1,N_2\}$ and $n\geq N$.

We will need, for this step only, a notation for the tunnels involved in the following computation. We write $\tau_{j(n)}^{d(s)} = (\D_n^{d(s)},\Lipp_n^{d(s)},\pi_n^{d(s)},\rho_n^{d(s)})$ and $\gamma_{j(n)}^{c(s)} = (\D_n^{c(s)},\Lipp_n^{c(s)},\rho_n^{c(s)},\pi_n^{c(s)})$.

By Definition (\ref{extent-def}) and \cite[Proposition 2.12]{Latremoliere14}, there exists $\eta_n \in \StateSpace\left(\A_{j(n)}^{c(s)}\right)$ such that $\Kantorovich{\Lipp_n^{c(s)}}(\mu,\eta_n) < \frac{\varepsilon}{3 (D l + 1)}$. Since $\varphi_{j(n)}^s$ is positive and unital, the map $\mu_n = \eta_n\circ\varphi_n^{s}$ is a state of $\A_{j(n)}^{d(s)}$. Again by Definition (\ref{extent-def}), there exists $\nu_n \in \StateSpace(\A^{d(s)})$ with $\Kantorovich{\Lipp_n^{d(s)}}(\nu_n,\mu_n) < \frac{\varepsilon}{3 (l + 1)}$.

Moreover, by Definition (\ref{targetset-def}), since $b_n \in \targetsettunnel{\tau_{j(n)}^{d(s)}}{a}{l}$, there exists $d_n \in \D_n^{d(s)}$ with $\Lipp_n^{d(s)}(d_n) \leq l$ such that $\pi_n^{d(s)}(d_n) = a$ and $\rho_n^{d(s)}(d_n) = b_n$. There also exists $d'_n\in \D_{n}^{c(s)}$ such that $\Lipp_n^{c(s)}(d'_n) \leq D l$, $\pi_n^{c(s)}(d'_n) = f_n$ and $\rho_n^{d(s)}(d'_n) = \varphi^s_n(b_n)$.

We then compute
\begin{equation}\label{main-thm-eq-1}
\begin{split}
  |\mu(\varphi^s(a)) - \nu_n(a)| &\leq |\mu(\varphi^s(a)) - \mu(f_n)| + |\mu(f_n) - \nu_n(a)| \\
  &\leq \|\varphi^s(a) - f_n\|_\A + |\mu(f_n) - \nu_n(a)| \\
  &\leq \frac{\varepsilon}{3} + |\mu(f_n) - \nu_n(a)| \\
  &\leq \frac{\varepsilon}{3} + |\mu(f_n) - \mu_n(b_n)| + |\mu_n(b_n) - \nu_n(a)|\\
  &\leq \frac{\varepsilon}{3} + |\mu(f_n) - \eta_n(\varphi_n^s(b_n))| + |\mu_n\circ\rho_n^{d(s)}(d_n) - \nu_n\circ\pi_n^{d(s)}(d_n)| \\
  &\leq \frac{\varepsilon}{3} + |\mu\circ\pi_n^{c(s)}(d'_n) - \eta_n\circ\rho_n^{c(s)}(d'_n)| + l \Kantorovich{\Lipp_n^{d(s)}}(\nu_n,\mu_n)\\
  &\leq \frac{\varepsilon}{3} + D l \Kantorovich{\Lipp_n^{c(s)}}(\mu,\eta_n) + l \Kantorovich{\Lipp_n^{d(s)}}(\nu_n,\mu_n)\\
  &\leq \frac{\varepsilon}{3} + \frac{D l \varepsilon}{3 D l} + \frac{l \varepsilon}{3 l} = \varepsilon \text{.}
\end{split}
\end{equation}

Now, since $a \geq 0$, for any $n\geq N$, we have $\nu_n(a) \geq 0$. We therefore conclude that $\mu(\varphi^s(a)) \geq  0$, as it is the limit of the sequence of nonnegative numbers $(\nu_n(a))_{n\in\N}$ by Expression (\ref{main-thm-eq-1}). As $\mu\in\StateSpace(\A^{c(s)})$ was an arbitrary state, we conclude that $\varphi^s(a) \geq 0$.

By continuity of $\varphi^s$, since the space of positive elements in $\A^{c(s)}$ is closed in norm, and by assumption that any positive element of $\A^{d(s)}$ is the limit of positive elements in $\dom{\Lip^{d(s)}}$, we conclude that $\varphi^s$ is a positive map, as claimed.

\begin{step}
If, for some $s\in \mathcal{S}$ and for all $n\in\N$, the map $\varphi_n^s$ is a unital *-homomorphism, then $\varphi$ is a unital *-homomorphism; in particular $\opnorm{\varphi}{\A}{\B} = 1$.
\end{step}

Fix $s\in\mathcal{S}$ such that for all $n\in\N$, the map $\varphi_n^s$ is a unital *-homomorphism. For this step, set $D = D(s)$.

We begin by proving that $\varphi^s$, restricted to $\dom{\Lip_\A}$, is a Jordan-Lie morphism. Let $a,a' \in \dom{\Lip^{d(s)}}$ and $l \geq \max\{\Lip^{d(s)}(a),\Lip^{d(s)}(a')\}$. For each $n\in\N$, let $b_n \in \targetsetforward{j(n)}{a'}{l,D}$ and $b_n'\in\targetsetforward{j(n)}{a'}{l,D}$. By construction, $(b_n)_{n\in\N}$ converges to $\varphi^s(a)$ while $(b'_n)_{n\in\N}$ converges to $\varphi^s(a')$. On the other hand, by Lemma (\ref{targetset-forward-JL-morphism-lemma}) (in particular, using the notations therein), for all $n\in\N$, we have
\begin{equation*}
\Jordan{b_n}{b_n'}\in\targetsetforward{j(n)}{\Jordan{a}{a'}}{M(\|a\|_\A,\|a'\|_\A,l,l),D}
\end{equation*}
so $\left(\Jordan{b_n}{b_n'}\right)_{n\in\N}$ converges to $\varphi^s\left(\Jordan{a}{a'}\right)$. By uniqueness of the limit, we conclude
\begin{equation*}
\varphi^s\left(\Jordan{a}{a'}\right) = \Jordan{\varphi^s(a)}{\varphi^s(a')} \text{.}
\end{equation*}
The same results holds for the Lie product. Thus $\varphi^s$ restricted to $\dom{\Lip^{d(s)}}$ is a Jordan-Lie morphism. By continuity, $\varphi^{s}$ is a Jordan-Lie morphism from $\sa{\A^{d(s)}}$ to $\sa{\A^{c(s)}}$.

So $\varphi^{s}$ is a unital Jordan-Lie morphism from $\sa{\A^{d(s)}}$ to $\sa{\A^{c(s)}}$.

We note that by construction, $\varphi^s(a^\ast) = \varphi^s(a)^\ast$.

Let now $a,b \in \A^{d(s)}$. We fist note that
\begin{align*}
\varphi^s(\Jordan{a}{b}) &= \varphi^s(\Jordan{\Re a}{\Re b}) - \varphi^s(\Jordan{\Im a}{\Im b}) + i\left(\varphi^s(\Jordan{\Re a}{\Im b}) + \varphi^s(\Jordan{\Im a}{\Re b}) \right) \\
&= \Jordan{\varphi^s(\Re a)}{\varphi^s(\Re b)} - \Jordan{\varphi^s(\Im a)}{\varphi^s(\Im b)} \\ 
&\quad+ i\left(\Jordan{\varphi^s(\Re a)}{\varphi^s(\Im b)} + \Jordan{\varphi^s(\Im a)}{\varphi^s(\Re b)} \right)\\
&= \Jordan{\varphi^s(a)}{\varphi^s(b)} \text{.}
\end{align*}
The same computation holds for the Lie product. Therefore
\begin{align*}
\varphi^s(ab) &= \varphi^s(\Re(ab)) + i\varphi^s(\Im(ab)) \\
&= \varphi^s(\Jordan{a}{b}) + i\varphi^s(\Lie{a}{b}) \\
&= \Jordan{\varphi^s(a)}{\varphi^s(b)} + i\Lie{\varphi^s(a)}{\varphi^s(b)}\\
&= \varphi^s(a) \varphi^s(b) \text{.}
\end{align*}
Thus $\varphi$ is a unital *-homomorphism. This concludes our step.

This concludes our step and our theorem.
\end{proof}

\begin{remark}
We note that the morphisms obtained by Theorem (\ref{main-thm}) are certainly not unique in general. There are many different choices being made in their construction: the choice of the tunnels and the choices of subsequences of target sets from compactness. The proof shows how all these choices can be made in a coherent fashion to preserve composition, and ensure that the limit of identity automorphisms is the identity --- thus preserving inverse relationships.
\end{remark}

\begin{remark}
We need an additional assumption regarding the density of the positive elements with finite L-seminorm within the positive elements of the underlying C*-algebra for Item (1) of Theorem (\ref{main-thm}) which does not appear for the rest of the theorem, and in particular for Item (3) regarding *-homomorphisms.

This assumption is needed for the proof of Item (1) as we have no multiplicative property for this step, and not for Item (3) because *-homomorphisms are automatically positive. We note that for classical metric space, this special assumption is always satisfied because the truncation of a Lipschitz function is Lipschitz. We also note that when L-seminorms are obtained from ergodic group actions as in \cite{Rieffel98a}, which is an often used construction, then again this special property is met, using a standard regularization argument akin to the proof of the density of the domain of such L-seminorms in \cite{Rieffel98a}. However, in general, the situation can be more involved and the Leibniz property does not seem to ensure that this special condition will always hold.
\end{remark}

Minor changes to our proof of Theorem (\ref{main-thm}) would lead to a similar theorem for the quantum propinquity, or in general, for the dual propinquity constructed from journeys as in \cite{Latremoliere13b}. Such minor modifications may be used to gain additional insights on the properties of our functor.

We now have proven our main result. We will turn to its applications.

\section{Applications}

As a first observation, we apply Theorem (\ref{main-thm}) to the construction of a single *-homomorphism. The result is interesting because it is not necessarily easy to construct morphisms between C*-algebras, yet this first corollary shows a new method based on metric approximations.

\begin{corollary}\label{one-morphism-cor}
Let $F$ be an admissible function. Let $(\A,\Lip^\A)$ and $(\B,\Lip^\B)$ be two  {\Qqcms{F}s}, and let $(\A_n,\Lip^\A_n)_{n\in\N}$ and $(\B_n,\Lip^\B_n)_{n\in\N}$ be two sequences of {\Qqcms{F}s} such that
\begin{equation*}
\lim_{n\rightarrow\infty} \dpropinquity{F}((\A_n,\Lip^\A_n),(\A,\Lip^\A)) = \lim_{n\rightarrow\infty} \dpropinquity{F}((\B_n,\Lip^\B_n),(\B,\Lip^\B)) = 0 \text{.}
\end{equation*}

Let $D > 0$. If for each $n\in\N$, there exists a Lipschitz morphism $\varphi_n : \A_n \rightarrow \B_n$ such that:
\begin{enumerate}
\item $\forall n\in\N \quad \Lip^\B_n\circ\varphi_n \leq D \Lip^\A_n$,
\item $\opnorm{\varphi_n}{\A_n}{\B_n} \leq D$,
\end{enumerate}
then there exists a Lipschitz morphism $\varphi : \A \rightarrow \B$ such that:
\begin{enumerate}
\item $\opnorm{\varphi}{\A}{\B} \leq D$
\item $\Lip^\B\circ\varphi \leq D \Lip^\A$.
\end{enumerate}
If moreover $(\A_n,\Lip_n^\A) = (\B_n,\Lip_n^\B)$ for all $n\in\N$ then
\begin{equation*}
\liminf_{n\rightarrow\infty} \KantorovichLength{\Lip_n^\A}(\varphi_n) \leq \KantorovichLength{\Lip_\A}(\varphi) \leq \limsup_{n\rightarrow\infty} \KantorovichLength{\Lip_\A}(\varphi_n) \text{.}
\end{equation*}

If furthermore, $\varphi_n$ is a bijection for infinitely many $n\in\N$, then $\varphi$ may be chosen to be a bijection as well.
\end{corollary}

\begin{remark}
  We will see that the proof of Corollary (\ref{one-morphism-cor}), even when $\A_n=\A=\B_n=\B$ for all $n\in\N$, the morphism $\varphi$ is obtain as a sort of limit of a subsequence of $(\varphi_n)_{n\in\N}$, thanks to a compactness argument. Moreover, the compactness of target sets can be traced, eventually, to Kadison functional representation (which uses Hahn-Banach theorem) and the Banach-Alaoglu theorem (which uses Tychonoff theorem), so with our current arguments, the particular subsequence of $(\varphi_n)_{n\in\N}$ which allows us to construct $\varphi$ is obtain implicitly, relying on the axiom of choice. Thus the relationship between $\varphi$ and $(\varphi_n)_{n\in\N}$ is not explicit, let alone in the general framework of Corollary (\ref{one-morphism-cor}).
\end{remark}

\begin{proof}
We first work with the category on two objects  and one morphism between them, in addition to both identity maps, written as:
\begin{equation*}
  \left\{ \{ s_a, s, s_b  \}, \{ a,b \}, c, d, \circ   \right\}
\end{equation*}
with $s_a$ and $s_b$ the respective identities  of $a$ and $b$, and with $s : a \rightarrow b$, while:
\begin{equation*}
  d(s_a) = c(s_a) = d(s) = a\text{ and }d(s_b) = c(s_b) = c(s) = b
\end{equation*}
and the usual composition $\circ$ between all these maps. 

If, for all $n\in\N$, we set $(\A_n^a,\Lip_n^a) = (\A_n,\Lip_n^\A)$, $(\A_n^b,\Lip_n^b) = (\B_n,\Lip_n^\B)$, if $\varphi^s_n = \varphi_n$ and if we set $(\A^a,\Lip^a) = (\A,\Lip^\A)$ and $(\A^b,\Lip^b) = (\B,\Lip^\B)$, then our corollary follows from Theorem (\ref{main-thm})

If $\varphi_n$ is invertible for infinitely many $n\in\N$, then up to extracting a subsequence, we may as well assume that $\varphi_n$ is invertible for all $n\in\N$. Then we Theorem (\ref{main-thm}) applies again, using this time the category with two objects, their identity, and one isomorphism and its inverse between them.
\end{proof}

As an obvious remark, we point out that if we are given two convergent sequences $(\A_n,\Lip^\A_n)_{n\in\N}$ and $(\B_n,\Lip_n^\B)_{n\in\N}$ of {\gQqcms s} which are entry-wise fully quantum isometric, then by \cite{Latremoliere13b}, their respective limits $(\A,\Lip^\A)$ and $(\B,\Lip^\B)$ are also fully quantum isometric
\begin{equation*}
0 \leq \propinquity{}((\A,\Lip^\A),(\B,\Lip^\B)) = \lim_{n\rightarrow\infty} \propinquity{}((\A_n,\Lip_n^\A),(\B,\Lip_n^\B)) = \lim_{n\rightarrow\infty} 0 = 0 \text{.}
\end{equation*}
This coincidence property was a strong motivation behind the construction of the propinquity, and is proven using in part the quasi-Leibniz property. Thus our present Theorem (\ref{main-thm}) provides another illustration of the same principle that using the quasi-Leibniz property and the definition of the propinquity enables us to build *-homomorphisms between C*-algebras, albeit Theorem (\ref{main-thm}) applied in the single morphism case is only new and interesting for Lipschitz morphisms which are not full quantum isometries.

\medskip

We now turn to the main motivation for this paper: the existence of a nontrivial action of a group on the limit of a sequence of {\gQqcms s}, when the group is the limit of groups acting on each term of the sequence. One way to see the importance of this result is with a sight on mathematical physics applications: a model obtained as a limit for the propinquity of covariant models for some group will bear the same covariance.

There are in fact different manners to formalize this problem. We will study the equivariant propinquity in a subsequent paper \cite{Latremoliere18b}, to keep the present work of manageable length, and also because we take the perspective that this is a work on the properties of the Gromov-Hausdorff propinquity. 

We introduce a natural notion of compact, metric monoid convergence. We will use the term monoid to mean a set endowed with an associative binary operation which admits an identity element (a unit). A \emph{metric monoid} $(G,\delta_G)$ is a monoid $G$ endowed with a distance $\delta_G$ which induces a topology on $G$, and for which the multiplication of $G$ is continuous. A morphism of metric monoid will be, for our purpose, a continuous algebraic morphism.

Now, we define a notion of convergence for compact metric monoids, which we study in some details In \cite{Latremoliere18b,Latremoliere18c}. The idea is to find a pair of maps which are as ``close as possible'' to being a compact monoid isomorphism and its inverse.

\begin{definition}\label{near-iso-def}
  Let $(G_1,\delta_1)$ and $(G_2,\delta_2)$ be two metric monoids and $\varepsilon > 0$. An ordered pair $\varsigma_1:G_1\rightarrow G_2$ and $\varsigma_2 : G_2\rightarrow G_1$ of functions is called \emph{$\varepsilon$-near isometric isomorphism from $(G_1,\delta_1)$ to $(G_2,\delta_2)$} when, for all $\{j,k\} = \{1,2\}$ and $g,g' \in G_j$:
  \begin{enumerate}
  \item $\delta_k(\varsigma_j(g)\varsigma_j(g'),\varsigma_j(g g')) \leq \varepsilon$,
  \item $\delta_j(\varsigma_k(\varsigma_j(g)),g) \leq \varepsilon$,
  \item $\left|\delta_k(\varsigma_j(g),\varsigma_j(g')) - \delta_j(g,g')\right| \leq \varepsilon$,
  \item $\delta_k(\varsigma_j(e_j),e_k) \leq \varepsilon$.
  \end{enumerate}
\end{definition}

\begin{definition}\label{group-GH-def}
A sequence $(G_n,\delta_n)_{n\in\N}$ of compact monoids \emph{converges, in the sense of the covariant Gromov-Hausdorff topology,} to a compact monoid $(G,\delta)$ when, for all $\varepsilon > 0$, there exists $N\in\N$ such that if $n\geq N$ then there exists a $\varepsilon$-near isometric isomorphism from $G_n$ to $G$.
\end{definition}

We refer to \cite{Latremoliere18b} for a more detailed presentation of the basic properties of the notion of convergence introduced in Definition (\ref{group-GH-def}); we only record here the following:
\begin{enumerate}
  \item there exists a pseudo-metric on the class of compact metric monoids for which convergence of a sequence of compact metric monoids is equivalent to the convergence of that sequence in the sense of Definition (\ref{group-GH-def}),
  \item if a sequence of compact metric monoid converges to both $(G,\delta_G)$ and $(H,\delta_H)$ in the sense of Definition (\ref{group-GH-def}), then there exists $\pi : G \rightarrow H$ which is a bijection, an isometry, and an algebra isomorphism from $(G,\delta_G)$ to $(H,\delta_H)$.
\end{enumerate}

A special case of obvious interest is given by converging sequences of closed subgroups for the Hausdorff metric induced by a continuous length function on a compact group: the following proposition shows that such a sequence would also converge in the sense of Definition (\ref{group-GH-def}) even though we do not a priori assume more than convergence in a metric sense, with no algebraic condition.

\begin{proposition}\label{upsilon-Haus-prop}
Let $(G,\delta)$ be a compact metric group, with $\delta$ invariant by left and right translations. If $H\subseteq G$ is a closed subset of $G$ and $(H_n)_{n\in\N}$ is a sequence of closed subgroups of $G$ converging to $H$ for the Hausdorff distance $\Haus{\delta}$ induced by $\delta$ on the closed subsets of $G$, then $H$ is a closed subgroup of $G$ and $(H_n,\delta)_{n\in\N}$ converges to $(H,\delta)$ in the sense of Definition (\ref{group-GH-def}).
\end{proposition}

\begin{remark}
If $\ell$ is a length function over a group $G$ such that $\ell(g^{-1} h g) = \ell(h)$ for all $g,h \in G$, then the distance $\delta :g,h \in G\mapsto \ell(h^{-1} g)$ is both left and right invariant, i.e. translation invariant.
\end{remark}

\begin{proof}
It is immediate that the unit of $G$ belongs to $H$. Let $g,h \in H$. There exist sequences $(g_n)_{n\in\N}$ and $(h_n)_{n\in\N}$ in $G$ with $g_n, h_n \in H_n$ for all $n\in\N$, converging respectively to $g$ and $h$. Now $g_n h_n^{-1} \in H_n$ (as $H_n$ is a group) for all $n\in\N$, and by continuity of the multiplication and the inverse map of $G$, and by basic properties of the Hausdorff distance, $gh^{-1} \in H$ and thus $H$ is a subgroup of $G$. Of course, $H$ is closed as a compact subset of a Hausdorff space.

For any closed subset $F$ of $G$ and $g\in G$, we choose an element $p_F(g) \in F$ such that
\begin{equation*}
\delta(g,p_F(g)) = \min \left\{ \delta(g,h) : h \in F \right\}\text{.}
\end{equation*}
The existence of such an element is guaranteed by compactness; our choice is arbitrary: in particular, we have no expectation of any regularity for the function $p_F$ thus defined.

For all $n\in\N$, set
\begin{equation*}
\varkappa_n : g \in H \mapsto p_{H_n}(g) \text{ and }\varpi_n : g \in H_n \mapsto p_{H}(g) \text{.}
\end{equation*}

We immediately note that $\varkappa_n$ and $\varpi_n$ both fix the identity element of $G$ for all $n\in\N$.

Let $N\in\N$ such that, for all $n\geq N$, we have $\Haus{\delta}(H_n,H) < \frac{\varepsilon}{3}$. Let $n\geq N$. Note that by definition of $\Haus{\delta}$, we have $\max\{\delta(\varkappa_n(g),g),\delta(\varpi_n(h),h)\} < \frac{\varepsilon}{3}$ for all $g\in H$ and $h \in H_n$. It is a matter of routine computation that $(\varsigma_n,\varkappa_n)$ is a $\varepsilon$-near isometric isometry from $(H_n,\delta)$ to $(H,\delta)$ as desired.
\end{proof}

We are now ready to prove our main application of Theorem (\ref{main-thm}).

\begin{theorem}\label{semigroup-thm}
Let $(G,\delta_G)$ be a compact metric monoid. Let $(G_n,\delta_n)_{n\in\N}$ be a sequence of compact metric monoids, converging to $(G,\delta_G)$. Let $(\varepsilon_n)_{n\in\N}$ be a sequence of positive numbers with $\lim_{n\rightarrow\infty}\varepsilon_n = 0$ and such that for each $n\in\N$, there exists a $\varepsilon_n$-near isometric isomorphism $(\varsigma_n,\varkappa_n)$ from $(G_n,\delta_n)$ to $(G,\delta_G)$.

Let $D : G \rightarrow [0,\infty)$ be a continuous function. Let $K : [0,\infty) \rightarrow [0,\infty)$ be continuous with $K(0) = 0$.

Let $\mathcal{C}$ be a nonempty class of {\Qqcms{F}s} for some admissible function $F$, and let $\mathcal{T}$ be a class of tunnels compatible with $\mathcal{C}$. Let $(\A_n,\Lip_n)_{n\in\N}$ be a sequence in $\mathcal{C}$, such that for each $n\in\N$, there exists an action $\alpha_n$ by Lipschitz linear endomorphisms of $G_n$ on $\A_n$, where
\begin{equation*}
\forall n\in\N,\; g \in G_n \quad \opnorm{\alpha_n^g}{}{\A_n} \leq D(\varsigma_n(g))\text{ and }\Lip_n\circ\alpha_n^g \leq D(\varsigma_n(g)) \Lip_n \text{,}
\end{equation*}
and
\begin{equation*}
\forall n \in \N, \; g,g' \in G_n , \; a\in\sa{\A_n} \quad \left\| \alpha_n^g(a) - \alpha_n^{g'}(a) \right\|_{\A_n} \leq K\left(\delta_n(g,g')\right) \Lip_n(a) \text{.}
\end{equation*}

If there exists $(\A,\Lip) \in \mathcal{C}$ such that
\begin{equation*}
\lim_{n\rightarrow\infty} \dpropinquity{\mathcal{T}}((\A_n,\Lip_n),(\A,\Lip)) = 0
\end{equation*}
then there exists a strongly continuous action $\alpha$ of the monoid $G$ by Lipschitz linear endomorphisms such that for all $g\in G$
\begin{equation}\label{group-cor-eq1}
\opnorm{\alpha^g}{}{\A} \leq D(g) \text{ and }\Lip\circ\alpha^g \leq D(g) \Lip\text{,}
\end{equation}
and for all $g,h \in G$
\begin{multline}\label{group-cor-eq2}
\liminf_{n\rightarrow\infty} \KantorovichDist{\Lip_n}{\A_n}{\alpha_n^{\varkappa_n(g)}}{\alpha_n^{\varkappa_n(h)}} \leq \KantorovichDist{\Lip}{\A}{\alpha^g}{\alpha^h} \\ \leq \limsup_{n\rightarrow\infty} \KantorovichDist{\Lip_n}{\A_n}{\alpha_n^{\varkappa_n(g)}}{\alpha_n^{\varkappa_n(h)}}\text{,}
\end{multline}
while for all $g,h \in G$, and $a\in\dom{\Lip}$, we have
\begin{equation*}
\left\|\alpha^g(a) - \alpha^h(a)\right\|_\A \leq \Lip(a) K(\delta_G(g,h)) \text{.}
\end{equation*}

If moreover for all $n\in\N$, the action $\alpha_n$ of the monoid $G_n$ is by positive unital linear maps (resp. by Lipschitz morphisms), then $\alpha$ is a strongly  continuous action of $G$ by positive unital linear maps (resp. by Lipschitz morphisms).
\end{theorem}

\begin{remark}
With the notations of Theorem (\ref{semigroup-thm}), if $G$ is a group, and $G_n$ is a group for all $n\in\N$, then of course $\opnorm{\alpha^g}{}{\A} = 1$ as $\alpha^g$ is a *-automorphism for all $g\in G$.
\end{remark}

\begin{remark}
We emphasize that while Theorem (\ref{semigroup-thm}) requires convergence of the quantum metric spaces and of the groups or monoids, there is no requirement of convergence involving the actions themselves.
\end{remark}

\begin{proof}
Since $G$ is a compact metric space, it is separable. Let $E$ be a countable dense subset of $G$. Let $H$ be the sub-monoid generated by $E$. Since $H$ consists of all the \emph{finite} products of elements of the countable set $E$, it is itself countable. 

As a monoid, $H$ is also a small category in a trivial manner (formally, the category is $(H,\{e\},c,d,\cdot)$ with $e$ the unit of $H$ (and $G$), the maps $c$ and $d$ the constant functions equal to $e$ and $\cdot$ the multiplication of $H$. With these notations, we would set $(\A^e,\Lip^e) = (\A,\Lip)$ and $(\A_n^e,\Lip_n^e) = (\A_n,\Lip_n)$ for all $n\in\N$, though we will not use this heavier notation here).

Let $g\in H$, $n\in\N$ and $a\in\dom{\Lip_n}$ with $\Lip_n(a)\leq 1$. We now compute
\begin{align*}
    \left\| \alpha_n^{\varkappa_n(g)}\circ\alpha_n^{\varkappa_n(g')}(a) -  \alpha_n^{\varkappa_n(g g')}(a) \right\|_{\A_n}    &= \left\| \alpha_n^{\varkappa_n(g) \varkappa_n(g')}(a) -  \alpha_n^{\varkappa_n(g g')}(a) \right\|_{\A_n}  \\
    &\leq K\left(\delta_n\left(\varkappa_n(g) \varkappa_n(g'), \varkappa_n(g g')\right)\right) \Lip_n(a) \\
    &\leq K\left(\delta_n\left(\varkappa_n(g) \varkappa_n(g'), \varkappa_n(g g')\right)\right) \text{.} 
\end{align*}

By Definition (\ref{near-iso-def}), the sequence $\left(\delta_n\left(\varkappa_n(g) \varkappa_n(g'), \varkappa_n(g g')\right)\right)_{n\in\N}$ converges to $0$. So
\begin{equation*}
\lim_{n\rightarrow\infty} K\left(\delta_n\left(\varkappa_n(g) \varkappa_n(g'), \varkappa_n(g g')\right)\right) = 0 \text{.}
\end{equation*} 

We also record that $\alpha_n^{e_n}$, for $e_n \in G_n$ the unit of $G_n$, is the identity map.

Thus we fit the hypothesis of Theorem (\ref{main-thm}). Therefore, there exists an action $\alpha$ of $H$ on $\A$, with the property that, for all $a\in\sa{\A}$ with $\Lip(a)\leq 1$ and for all $g,h \in H$
\begin{equation*}
\|\alpha^g(a) - \alpha^h(a)\|_\A \leq \limsup_{n\rightarrow\infty} K(\delta_n(\varkappa_n(g),\varkappa_n(h))) \leq K(\delta_G(g,h)) \text{,}
\end{equation*}
and in addition, Expressions (\ref{group-cor-eq1}) and (\ref{group-cor-eq2}) hold for $g \in H$. Note that $\alpha$ is a monoid action, i.e. in particular the unit of $G$ acts as the identity.

It thus immediately follows by homogeneity that for all $a\in\dom{\Lip}$ and $g \in H$
\begin{equation*}
\|\alpha^g(a) - \alpha^h(a)\|_\A \leq \Lip(a) K(\delta_G(g,h)) \text{.}
\end{equation*}

Therefore $g \in H \mapsto \alpha^g(a)$ is uniformly continuous for all $a\in\dom{\Lip}$ since $\lim_0 K = 0$,  and thus it can be extended uniquely to a uniformly continuous function $g \in G \mapsto \alpha^g(a)$, with the additional property that for all $g,h \in G$
\begin{equation*}
\left\| \alpha^g (a) - \alpha^h(a)\right\|_\A \leq \Lip(a) \diam{\A}{\Lip} K(\delta_G(g,h))
\end{equation*}
by continuity of $K$ (and $\delta_G$).

It is straightforward that for all $a\in\dom{\Lip}$:
\begin{enumerate}
\item for all $g,h \in G$, we have by continuity of the multiplication on $G$ that $\alpha^g\circ\alpha^h(a) = \alpha^{gh}(a)$,
\item for all $g \in G$, the lower semi-continuity of $\Lip$ and the continuity of $D$ gives us $\Lip\circ\alpha^g(a) \leq D(g) \Lip(a)$,
\item for all $g\in G$, we also have by continuity that $\|\alpha^g(a)\|_\A  \leq D(g) \|a\|_\A$,
\item if for $g\in G$ and if all $n\in\N$, the map $\alpha_n^g$ is positive, or a *-endomorphism, so is $\alpha^g$.
\end{enumerate}
In particular, Expression (\ref{group-cor-eq1}) holds for all $g\in G$.

Therefore, $\alpha$ is a monoid action of $G$ via Lipschitz linear maps (and, if $\alpha_n$ acts by Lipschitz morphisms or acts by positive linear maps, for all $n\in\N$, then so does $\alpha$).

We now check the other properties of this action $\alpha$.

Let $g,g' \in G$. Let $\varepsilon  > 0$, and let $R > 0$ be some upper bound for the convergent sequence $(\diam{\A_n}{\Lip_n})_{n\in\N}$. Since $K$ is continuous and $K(0) = 0$, there exists $\delta > 0$ such that for $t \in [0,\delta)$, we have $K(t) < \frac{\varepsilon}{4}$.

Let $h,h' \in H$ with $\delta_G(g,h) < \frac{\delta}{3}$ and $\delta_G(g',h') < \frac{\delta}{3}$.

Now, there exists $N\in\N$ such that for all $n\geq N$, we have $\varepsilon_n < \frac{\delta}{6}$. Let $n \geq N$. By Definition (\ref{near-iso-def}), we have
\begin{equation*}
\left|\delta_n(\varkappa_n(g),\varkappa_n(h)) - \delta_G(g,h)\right| < \frac{\delta}{3} \text{.}
\end{equation*}
Hence $\delta_n(\varkappa_n(g),\varkappa_n(h)) < \delta$, and  similarly $\delta_n(\varkappa_n(g'),\varkappa_n(h')) < \delta$.

We now compute, for any $a\in\dom{\Lip_n}$ with $\Lip_n(a) \leq 1$
\begin{multline*}
  \left| \left\|\alpha_n^{\varkappa_n(g)}(a) - \alpha_n^{\varkappa_n(g')}(a)\right\|_{\A_n} - \left\|\alpha_n^{\varkappa_n(h)}(a) - \alpha_n^{\varkappa_n(h')}(a)  \right\|_{\A_n} \right| \\
  \begin{split}
    &\leq \left\|\alpha_n^{\varkappa_n(g)}(a) - \alpha_n^{\varkappa_n(h)}(a) \right\|_{\A_n} 
    + \left\|\alpha_n^{\varkappa_n(h')}(a) - \alpha_n^{\varkappa_n(g')}(a) \right\|_{\A_n} \\ 
    &\leq K(\delta_n(\varkappa_n(g),\varkappa_n(h))) + K(\delta_n(\varkappa_n(g'),\varkappa_n(h')))  \leq \frac{\varepsilon}{2} \text{.}
  \end{split}
\end{multline*}
Therefore
\begin{equation*}
\limsup_{n\rightarrow\infty}\KantorovichDist{\Lip_n}{}{\alpha_n^{\varkappa_n(g)}}{\alpha_n^{\varkappa_n(g')}} \leq \limsup_{n\rightarrow\infty}\KantorovichDist{\Lip_n}{}{\alpha_n^{\varkappa_n(h)}}{\alpha_n^{\varkappa_n(h')}} + \frac{\varepsilon}{2} \text{,}
\end{equation*}
and
\begin{equation*}
\liminf_{n\rightarrow\infty}\KantorovichDist{\Lip_n}{}{\alpha_n^{\varkappa_n(g)}}{\alpha_n^{\varkappa_n(g')}} \geq \liminf_{n\rightarrow\infty}\KantorovichDist{\Lip_n}{}{\alpha_n^{\varkappa_n(h)}}{\alpha_n^{\varkappa_n(h')}} - \frac{\varepsilon}{2} \text{.}
\end{equation*}

Now
\begin{multline*}
  \left| \left\| \alpha^g(a) - \alpha^{g'}(a) \right\|_\A - \left\| \alpha^h(a) - \alpha^{h'}(a) \right\|_\A \right| \\
  \begin{split}
    &\leq \left\| \alpha^g(a) - \alpha^h(a) \right\|_\A + \left\| \alpha^{g'}(a) - \alpha^{h'}(a) \right\|_\A\\
    &\leq K(\delta_G(g,h)) + K(\delta_G(g',h')) < \frac{\varepsilon}{2} \text{.}
  \end{split}
\end{multline*}

Thus
\begin{align*}
\left\| \alpha^{g}(a) - \alpha^{g'}(a) \right\|_\A &\leq \frac{\varepsilon}{2} + \left\| \alpha^h(a) - \alpha^{h'}(a) \right\|_\A \\
&\leq \frac{\varepsilon}{2} + \limsup_{n\rightarrow\infty} \KantorovichDist{\Lip_n}{}{\alpha^{\varkappa_n(h)}}{\alpha^{\varkappa_n(h')}} \text{ since $h,h' \in H$,}\\
&\leq \frac{\varepsilon}{2} + \frac{\varepsilon}{2} + \limsup_{n\rightarrow\infty} \KantorovichDist{\Lip_n}{}{\alpha^{\varkappa_n(g)}}{\alpha^{\varkappa_n(g')}} \\
&= \varepsilon + \limsup_{n\rightarrow\infty} \KantorovichDist{\Lip_n}{}{\alpha^{\varkappa_n(g)}}{\alpha^{\varkappa_n(g')}} \text{,}
\end{align*}
and therefore
\begin{align*}
\KantorovichDist{\Lip}{}{\alpha^g}{\alpha^{g'}} \leq \limsup_{n\rightarrow\infty} \KantorovichDist{\Lip_n}{}{\alpha^{\varkappa_n(g)}}{\alpha^{\varkappa_n(g')}} + \varepsilon \text{.}
\end{align*}

Similarly
\begin{align*}
\KantorovichDist{\Lip}{}{\alpha^g}{\alpha^{g'}} \geq \liminf_{n\rightarrow\infty} \KantorovichDist{\Lip_n}{}{\alpha^{\varkappa_n(g)}}{\alpha^{\varkappa_n(g')}} - \varepsilon \text{.}
\end{align*}

Since $\varepsilon > 0$ is arbitrary, we conclude that for all $g,g' \in G$
\begin{multline*}
\liminf_{n\rightarrow\infty} \KantorovichDist{\Lip_n}{}{\alpha_n^{\varkappa_n(g)}}{\alpha_n^{\varkappa_n(g')}} \leq \KantorovichDist{\Lip}{}{\alpha^g}{\alpha^{g'}} \\ \leq \limsup_{n\rightarrow\infty} \KantorovichDist{\Lip_n}{}{\alpha^{\varkappa_n(g)}}{\alpha^{\varkappa_n(g')}} \text{.}
\end{multline*}
Thus Expression (\ref{group-cor-eq2}) now holds for all $g,g' \in G$.

Let now $a\in\sa{\A}$, $\varepsilon > 0$ and $g \in G$. Since  $D$ is continuous, let $\eta > 0$ such that if $\delta_G(g,h) < \eta$, then $|D(g) - D(h)| < \varepsilon$. 

By density, there exists $a' \in \dom{\Lip}$ such that $\|a-a'\|<\frac{\varepsilon}{3(D(g) + \varepsilon)}$.  Now, there exists $\eta_2 > 0$ such that if $t \in [0,\eta_2)$ then $K(t) < \frac{\varepsilon}{3\max\{\Lip(a'),1\}}$. Let now $h\in G$ with $\delta_G(g,h) < \min\{\eta,\eta_2\}$. We compute
\begin{align*}
  \left\|\alpha^g(a) - \alpha^h(a)\right\|_\A &\leq \left\|\alpha^g(a-a')\right\|_\A + \left\|\alpha^g(a') - \alpha^h(a')\right\|_\A + \left\|\alpha^h(a - a')\right\|_\A \\
    &\leq D(g)\left\|a-a'\right\|_\A + \Lip(a') K(\delta_G(g,h)) + D(h)\left\|a-a'\right\|_\A \\
    &\leq \varepsilon\text{.}
\end{align*}

Therefore, for all $a\in\sa{\A}$, the function $g\in G\mapsto \alpha^g(a)$ is continuous, as desired.
\end{proof}

We now turn to an application of our work to certain examples. We can begin to address two related difficult questions using our work in this paper, and we will see that quantum metric geometric ideas might continue to prove helpful in their study.

Ergodic actions of compact groups on unital C*-algebras have been of great interests, and it has been challenging to determine, for a given compact group, even compact Lie groups, what are all the C*-algebras on which the group acts ergodically, outside of some special cases.

On the other hand, another difficult question in noncommutative metric geometry is to determine the closure of certain classes of {\gQqcms s} for the propinquity. For instance, it remains unclear what the closure of the finite dimensional {\gQqcms s} is at the time of this writing. When restricting our focus to finite dimensional algebras whose quantum metrics come from an ergodic action of compact groups, our present work will enable us to make significant advance is answering the closure question.

As a first simple observation, working with groups enables us to strengthen the conclusions of Theorem (\ref{semigroup-thm}), by obtaining actions by full quantum isometries. The following proposition will be used to this end in our next theorem on group actions.

\begin{proposition}\label{group-full-iso-prop}
If $\alpha$ is an action of a compact group $G$ on some {\gQqcms} $(\A,\Lip)$ by $1$-Lipschitz automorphisms (i.e. $\Lip\circ\alpha^g \leq  \Lip$ for all $g\in G$), then for all $g\in G$, the automorphism $\alpha^g$ is a full quantum isometry of $(\A,\Lip)$.
\end{proposition}

\begin{proof}
Let $g \in G$. By assumption, for all $a\in \sa{\A}$, we have $\Lip\circ\alpha^{g^{-1}}(a) \leq \Lip(a)$. Thus for all $a\in\dom{\Lip}$
\begin{equation*}
\Lip(a) = \Lip(\alpha^{g^{-1}}(\alpha^{g})) \leq \Lip(\alpha^{g}(a)) \text{.}
\end{equation*}

 Hence $\Lip\circ\alpha^g\leq\Lip\leq\Lip\circ\alpha^g$ and thus $\Lip\circ\alpha^g = \Lip$. Consequently, $\alpha^g$ is a full quantum isometry for any $g\in G$.
\end{proof}

As an important corollary of Theorem (\ref{semigroup-thm}), we obtain:
\begin{theorem}\label{group-thm}
Let $(G,\delta_G)$ be a compact metric group. Let $(G_n,\delta_n)_{n\in\N}$ be a sequence of compact metric groups, converging to $(G,\delta_G)$.  Let $(\varepsilon_n)_{n\in\N}$ be a sequence of positive numbers with $\lim_{n\rightarrow\infty}\varepsilon_n = 0$ and such that for each $n\in\N$, there exists a $\varepsilon_n$-near isometric isomorphism $(\varsigma_n,\varkappa_n)$ from $(G_n,\delta_n)$ to $(G,\delta_G)$.

Let $D : G \rightarrow [0,\infty)$ be a continuous function. Let $K : [0,\infty) \rightarrow [0,\infty)$ be continuous with $K(0) = 0$.

Let $\mathcal{C}$ be a nonempty class of {\Qqcms{F}s} for some admissible function $F$, and let $\mathcal{T}$ be a class of tunnels compatible with $\mathcal{C}$. Let $(\A_n,\Lip_n)_{n\in\N}$ be a sequence in $\mathcal{C}$, such that for each $n\in\N$, there exists an action $\alpha_n$ by Lipschitz automorphisms  of $G_n$ on $\A_n$, where
\begin{equation*}
\forall n\in\N,\; g \in G_n \quad \Lip_n\circ\alpha_n^g \leq D(\varsigma_n(g)) \Lip_n \text{,}
\end{equation*}
and
\begin{equation*}
\forall n \in \N, \; g,g' \in G_n , \; a\in\sa{\A_n} \quad \left\| \alpha_n^g(a) - \alpha_n^{g'}(a) \right\|_{\A_n} \leq K\left(\delta_n(g,g')\right) \Lip_n(a) \text{.}
\end{equation*}

If there exists $(\A,\Lip) \in \mathcal{C}$ such that
\begin{equation*}
\lim_{n\rightarrow\infty} \dpropinquity{\mathcal{T}}((\A_n,\Lip_n),(\A,\Lip)) = 0
\end{equation*}
then there exists a strongly continuous action $\alpha$ of the group $G$ by Lipschitz automorphisms such that for all $g\in G$
\begin{equation*}
\Lip\circ\alpha^g \leq D(g) \Lip\text{,}
\end{equation*}
and for all $g \in G$
\begin{equation*}
\liminf_{n\rightarrow\infty} \KantorovichLength{\Lip_n}\left(\alpha_n^{\varkappa_n(g)}\right) \leq \KantorovichLength{\Lip}\left(\alpha^g\right) \leq \limsup_{n\rightarrow\infty} \KantorovichLength{\Lip_n}\left(\alpha_n^{\varkappa_n(g)}\right)\text{,}
\end{equation*}
while for all $g,h \in G$, and $a\in\dom{\Lip}$, we have
\begin{equation*}
\left\|\alpha^g(a) - \alpha^h(a)\right\|_\A \leq \Lip(a)K(\delta_G(g,h)) \text{.}
\end{equation*}

Furthermore
\begin{enumerate}
\item if $D \leq 1$ then for all $g\in G$, the *-automorphism $\alpha^g$ is a full quantum isometry,
\item if for all $n\in\N$, the action $\alpha_n$ of $G_n$ on $\A_n$ is ergodic, then the action $\alpha$ of $G$ on $\A$ is also ergodic.
\end{enumerate}
\end{theorem}

\begin{proof}

\setcounter{step}{0}

As a corollary of Theorem (\ref{semigroup-thm}), there exists an action of $G$ on $\A$ by Lipschitz *-endomorphisms, only noting that  we can apply Theorem (\ref{main-thm}) to obtain an action of $G$ by *-automorphisms.

By Proposition (\ref{group-full-iso-prop}), if $D \leq 1$ then in fact, $\alpha^g$ is a full quantum isometry of $(\A,\Lip)$. 

We now turn to the ergodicity property. 

\begin{step}
We assume, for the rest of this proof, that for all $n\in\N$, the action $\alpha^n$ of $G_n$ on $\A_n$ is ergodic, namely that  $\{a\in\A : \forall g\in G\quad \alpha^g(a) = a\} = \C\unit_{\A_n}$. Our goal is to prove that $\alpha$ is also ergodic.
\end{step}

For this and the next few steps, we will need to be more specific about the construction of the action $\alpha$ provided by Theorem (\ref{main-thm}).

We start this proof at Step 4 of the proof of our main theorem (\ref{main-thm}). Thus, we have chosen a tunnel $\tau_n$ in $\mathcal{T}$ from $(\A,\Lip)$ to $(\A_n,\Lip_n)$ with
\begin{equation*}
\tunnelextent{\tau_n} \leq \propinquity{\mathcal{T}}((\A_n,\Lip_n),(\A,\Lip)) + \frac{1}{n+1}\text{,}
\end{equation*}
and we let $\gamma_n = \tau_n^{-1}$ for all $n\in\N$. To simplify notations, we have extracted a subsequence from $(\A_n,\Lip_n)_{n\in\N}$ which we still denote by $(\A_n,\Lip_n)_{n\in\N}$ such that for all $g\in G$, $a\in\dom{\Lip}$ and $l\geq \Lip(a)$
\begin{equation*}
\lim_{n\rightarrow\infty} \Haus{\|\cdot\|_\A}\left( \targetsetforward{n,g}{a}{l,1}, \{\alpha^g(a)\} \right) = 0 \text{,}
\end{equation*}
where $\targetsetforward{n,g}{a}{l,1}$ is a shorthand for $\targetsetforward{\tau_n,\alpha_n^{\varkappa_n(g)},\gamma_n}{a}{l,1}$.

Now, for each $n\in\N$, and $a\in\A_n$, we set
\begin{equation*}
\mathds{E}_n(a) = \int_{G_n} \alpha_n^{g}(a) \, d\lambda_n(g)
\end{equation*}
where $\lambda_n$ is the Haar probability measure on $G_n$. Of course, $\mathds{E}_n$ is the conditional expectation of $\A_n$ onto the fixed point algebra for $\alpha_n$ which is, for all $n\in\N$, reduced to $\C\unit_{\A_n}$ by assumption (indeed
\begin{equation*}
\alpha_n^h(\mathds{E}_n(a)) = \int_{G_n} \alpha_n^{hg}(a)\,d\lambda_n(g) = \int_{G_n} \alpha_n^{g}(a)\,d\lambda_n(gh^{-1}) = \mathds{E}_n(a)
\end{equation*}
since $\lambda_n$ is translation invariant.)

Similarly, let $\mathds{E} : a\in\A\mapsto\int_G \alpha^g(a)\, d\lambda(g)$. Our goal is thus to prove that the range of $\mathds{E}$ is $\C\unit_\A$.

Our first observation is that:
\begin{step} It is sufficient to show that if $a\in\dom{\Lip}$ and $\alpha^g(a) = a$ for all $g\in G$, then $a\in \R\unit_\A$. 
\end{step}

Indeed, should this be the case, then, first, for all $g\in G$ we have $\Lip(\alpha^g(a)) \leq \Lip(a)$ so $\alpha^g(a) \in \{ b \in \sa{\A} : \Lip(b)\leq\Lip(a),\|b\|_\A\leq\|a\|_\A\}$. As $\Lip$ is an L-seminorm, the latter set is compact for $\|\cdot\|_\A$. Thus, $\mathds{E}(a) = \int_G \alpha^g(a) \, d\lambda(g)$ lies in the same set, and in particular, it lies in $\dom{\Lip}$.

Since $\alpha^g(\mathds{E}(a)) = \mathds{E}(a)$ for all $g\in G$, we conclude that $\mathds{E}(a) \in \R\unit_\A$, by the assumption of our step.

Let $a\in\sa{\A}$ such that $\alpha^g(a) = a$ for all $g\in G$, so $\|a-\mathds{E}(a)\|_\A = 0$. By density of $\dom{\Lip}$, there exists $(a_n)_{n\in\N}$ in $\dom{\Lip}$ converging for $\|\cdot\|_\A$ to $a$.

By continuity, the function $g\in G\mapsto \|\alpha^g(a_n) - \alpha^h(a)\|_\A$ is pointwise convergent to $0$, and it is bounded. Hence by Lebesgues' dominated convergence theorem (as $\lambda$ is a bounded measure, so bounded functions are integrable), we conclude
\begin{align*}
\left\|\mathds{E}(a_n) - \mathds{E}(a)\right\|_\A &\leq \int_G \left\|\alpha^g(a_n) - \alpha^g(a)\right\|\,d\lambda(g) \\
&\xrightarrow{ n\rightarrow\infty} 0 \text{.}
\end{align*}
Hence $\mathds{E}(a) \in \R\unit_\A$ since $(\mathds{E}(a_n))_{n\in\N} \in \R\unit_\A^\N$ ans $\R\unit_\A$ is closed in  $\A$. As $a = \mathds{E}(a)$, we conclude that $a \in \R\unit_\A$.

Last, if $a\in\A$ and for all $g\in G$, we have $\alpha^g(a) = a$, then for all $g\in G$ we have $\alpha^g\left(\frac{1}{2}\left(a+a^\ast\right)\right) = \frac{1}{2}\left(a+a^\ast\right)$ and $\alpha^g(\frac{1}{2i}\left(a-a^\ast\right)) = \frac{1}{2i}\left(a-a^\ast\right)$, so $\frac{1}{2}\left(a+a^\ast\right), \frac{1}{2i}\left(a-a^\ast\right) \in \R\unit_\A$, and therefore $a\in \R\unit_\A$.

This concludes our proof of our first step.

\begin{step}
If $a\in\dom{\Lip}$ and for all $g\in G$, we have $\alpha^g(a) = a$, then $a\in\R\unit_\A$.
\end{step}

Let now $a\in\dom{\Lip}$ such that $\alpha^g(a) = a$ for all $g \in G$. Our goal is to prove $a\in\R\unit_\A$. Without loss of generality, we shall assume that $\Lip(a) \leq 1$.

Let $\varepsilon > 0$. As $K$ is continuous and $K(0) = 0$, there exists $\delta > 0$ such that if $t\in\R$, $|t|<\delta$ then $K(t) < \frac{\varepsilon}{7}$. Let $F \subseteq G$ be a finite, $\frac{\delta}{3}$-dense subset of $G$. 

Let $N_0 \in \N$ such that for all $n\geq N_0$, we have $\epsilon_n < \frac{\delta}{3}$. Note that by assumption, $\varkappa_n(F)$ is $\delta$-dense: if $g \in G_n$ then $\varpi(g) \in G$, so there exists $h \in F$ such that $\delta_G(g,h)<\frac{\delta}{3}$. Now
\begin{align*}
\delta_{G_n}(g,\varkappa_n(h)) &\leq \delta_{G_n}(g,\varkappa_n\circ\varpi_n(g)) + \delta_{G_n}(\varkappa_n\circ\varpi_n(g), \varkappa_n(h)) \\
&\leq \varepsilon_n + (\delta_G(\varpi_n(g),h) + \varepsilon_n)\\
&\leq 2\varepsilon_n + \frac{\delta}{3} \\
&< \delta\text{.}
\end{align*}

Let $N\in\N$ such that, for all $n\geq N$, we have $\tunnelextent{\tau_n} < \frac{\varepsilon}{7}$.

For each $g \in F$, let $N_g \in \N$ such that if $n \geq N_g$ and $c \in \targetsetforward{n,g}{a}{1,1}$ then $\|c - \alpha^g(a)\|_\A < \frac{\varepsilon}{7}$. Let $M = \max\{ N, N_g : g \in F \} \in \N$.

Let $n \geq M$.

Let $b_n \in \targetsettunnel{\tau_n}{a}{1}$. For all $g \in G$, let $c_{n,g} \in \targetsettunnel{\gamma_n}{\alpha_n^{\varkappa_n(g)}(b)}{1}$. Note that $c_{n,g} \in \targetsetforward{n,g}{a}{1,1}$. Thus by Proposition (\ref{targetset-norm-prop})
\begin{align*}
\left\|b_n - \alpha^{\varkappa_n(g)}_n(b_n)\right\|_{\A_n} &\leq 2\tunnelextent{\tau_n}\cdot 1 + \|a - c_{n,g}\|_\A \\
&\leq 2 \frac{\varepsilon}{7} + \|a-\alpha^g(a)\|_\A + \|\alpha^g(a) - c_{n,g}\|_\A \\
&< \frac{3\varepsilon}{7} \text{,}
\end{align*}
since $\|a-\alpha^g(a)\|_\A = 0$.

Let now $g \in G_n$, and $h \in \varkappa_n(F) \subseteq G_n$ such that $\delta_{G_n}(g, h) < \delta$. We estimate
\begin{align*}
\|b_n - \alpha_n^g(b_n)\|_{\A_n} &\leq \|b_n - \alpha_n^h(b_n)\|_{\A_n} + \|\alpha_n^h(b_n) - \alpha_n^g(b_n)\|_{\A_n} \\
&\leq \frac{3\varepsilon}{7} + K(\delta_{G_n}(g,h)) \\
&\leq \frac{4\varepsilon}{7} \text{.}
\end{align*}

Hence (as $\lambda_n(G_n) = 1$)
\begin{equation}\label{ergodic-main-eq}
\begin{split}
\left\|b_n - \mathds{E}_n(b_n)\right\|_{\A_n} &= \left\| b_n - \int_{G_n} \alpha_n^g(b_n) \, d\lambda_n(g) \right\|_{\A_n} \\
&= \left\| \int_{G_n} (b_n - \alpha_n^g(b_n)) \, d\lambda_n(g) \right\|_{\A_n}\\
&\leq \int_{G_n} \|b_n - \alpha_n^g(b_n)\|_{\A_n} \, d\lambda_n(g) \\
&\leq \frac{4\varepsilon}{7} \text{.}
\end{split}
\end{equation}

To ease notations, let $\mathds{E}_n(b_n) = t_n \unit_{\A_n}$ for all $n\in\N$. The sequence $(t_n)_{n\in\N}$ is bounded since, for all $n\in\N$
\begin{align*}
|t_n| &= \|\mathds{E}_n(b_n)\|_{\A_n} \leq \|b_n\|_{\A_n} \\
&\leq \tunnelextent{\tau_n} + \|a\|_\A \leq 1 + \|a\|_\A\text{ by Proposition (\ref{targetset-norm-prop}).}
\end{align*}

Let $f:\N\rightarrow\N$ be strictly increasing so that $(t_{f(n)})_{n\in\N}$ converges to some $t \in \R$. Let $N'\in\N$ such that, for $n\geq N'$, we have $|t-t_{f(n)}|< \frac{\varepsilon}{7}$.

Observe that by construction, $x\unit_{\A_n} \in \targetsettunnel{\tau_n}{x\unit_\A}{l}$ for any $x\in\R$, $l \geq 0$ and $n\in\N$.

Let $n\geq \max\{M,N'\}$. We compute
\begin{align*}
\|a - t\unit_\A\|_\A &\leq \|a - t_{f(n)}\unit_\A\|_\A + |t_{f(n)} - t| \\
&<\|a - t_{f(n)}\unit_\A\|_\A + \frac{\varepsilon}{7}\\
&\leq \frac{2\varepsilon}{7} + \|b_{f(n)} - t_{f(n)}\unit_{\A_{f(n)}}\|_{\A_{f(n)}} + \frac{\varepsilon}{7} \text{ by Proposition (\ref{targetset-norm-prop}),}\\
&= \|b_{f(n)} - \mathds{E}_{f(n)}(b_{f(n)})\|_{\A_{f(n)}} + \frac{3\varepsilon}{7} \\
&\leq \frac{4\varepsilon}{7} + \frac{3\varepsilon}{7} = \varepsilon \text{ by Eq. (\ref{ergodic-main-eq}).}
\end{align*}
Hence $a \in \R\unit_\A$ as desired.

This concludes our proof.
\end{proof}

We now apply Theorem (\ref{group-thm}) to determining certain closures of {\gQqcms s} under the propinquity.

\begin{corollary}
Let $G$ be a compact group, of unit $e \in G$, endowed with a continuous length function $\ell$ such that
\begin{equation*}
\forall g,h \in G \quad \ell(h g h^{-1}) = \ell(g) \text{.}
\end{equation*}

Let $\mathscr{G}$ be a nonempty collection of closed subgroups of $G$. We assume that $\mathscr{G}$ is closed for the Hausdorff distance $\Haus{\ell}$ induced by $\ell$.

Let $F$ be a permissible function. Let $\mathcal{C}$ be the class of all {\Qqcms{F}s} $(\A,\Lip)$ such that:
\begin{enumerate}
\item there exists $H\in \mathscr{G}$ and there exists a strongly continuous ergodic action $\alpha$ of $H$ on $\A$,
\item for all $a\in\sa{\A}$
\begin{equation*}
\Lip(a) \geq \sup\left\{\frac{\left\|a-\alpha^h(a)\right\|_\A}{\ell(h)} : h \in H\setminus\{e\} \right\} \text{,}
\end{equation*}
\item for all $g\in H$, the automorphism $\alpha^h$ is $1$-Lipschitz.
\end{enumerate}

The class $\mathcal{C}$ is not empty and closed for the dual propinquity.
\end{corollary}

\begin{remark}
  We do not assume that $G \in \mathscr{G}$.
\end{remark}

\begin{proof}
Now, let $H \in \mathscr{G}$. There exists an irreducible representation of $H$ on some $\C^n$ as $H$ is compact. Therefore, setting $\alpha^g : M \in \alg{M}_n \mapsto U^g M \left(U^{g}\right)^\ast$ for all $g\in H$ defines an ergodic action of $H$ on a full matrix algebra; it is easy to check that $(\alg{M}_n,a\mapsto\sup\left\{\frac{\|\alpha^g(a)-a\|_{\alg{M}_n}}{\ell(g)}:g\in H\setminus\{e\}\right\})$ lies in $\mathcal{C}$. 

Let $(\A,\Lip) \in \mathcal{C}$. There exists $H \in \mathscr{G}$ and an ergodic action $\alpha$ of $H$ on $(\A,\Lip)$ by $1$-Lipschitz automorphisms. For all $h\in H$ and $a\in\dom{\Lip}$, note that $\|a-\alpha^h(a)\|\leq \ell(h) \Lip(a)$ by construction; thus $\|\alpha^g(a) - \alpha^h(a)\|_\A \leq \Lip(a)\ell(gh^{-1})$ for all $g,h \in H$. Denoting $\delta (g,h) = \ell(gh^{-1})$ for all $g,h \in G$, we thus have
\begin{equation*}
\forall g,h \in H \quad \left\|\alpha^g(a) - \alpha^h(a)\right\|_\A \leq \Lip(a)\delta(g,h) \text{.}
\end{equation*}
By Proposition (\ref{group-full-iso-prop}), the automorphisms $\alpha^g$ are full quantum isometries for $\Lip$ for all $g\in G$.

Let now $(\A_n,\Lip_n)_{n\in\N}$ be any sequence in $\mathcal{C}$ converging to some $(\B,\Lip_\B)$ for $\propinquity{\mathcal{T}}$. For each $n\in\N$, there exists by assumption on $\mathcal{C}$, a group $H_n \in \mathscr{G}$, and an ergodic action $\alpha_n$ of $H_n$ on $\A_n$. Now, $\mathscr{G}$ is closed for $\Haus{\ell}$, so it is compact (as $G$ is compact). Thus there exists a subsequence $(H_{f(n)})_{n\in\N}$ converging to some closed subgroup $H \in \mathscr{G}$ of $G$ for $\Haus{\ell}$. By Proposition (\ref{upsilon-Haus-prop}), the sequence $(H_{f(n)})_{n\in\N}$ converges to $H$ in the sense of Definition (\ref{group-GH-def}) as well.

Thus, $(\A_{f(n)},\Lip_{f(n)})_{n\in\N}$ meets the hypothesis of Theorem (\ref{group-thm}). Therefore, there exists an ergodic action $\beta$ of $G$ on $(\B,\Lip_\B)$ by full quantum isometries, such that for all $a\in \sa{\B}$ and $g \in G$, we have $\|a - \beta^g(a)\|_\B \leq \ell(g) \Lip_\B(a)$. Thus $(\B,\Lip_\B) \in \mathcal{C}$, as desired.
\end{proof}

In particular, we see that:
\begin{corollary}\label{group-cor}
Let $\ell$ be a continuous length function on the $d$-torus $\T^d$ and $f : [0,\infty)\rightarrow [0,\infty)$. Let $F$ be a permissible function. Let $\mathcal{C}$ be the class of {\Qqcms{F}s} $(\A,\Lip)$ such that:
\begin{enumerate}
\item $\A$ is finite dimensional,
\item there exists an ergodic action $\alpha$ of a closed subgroup $H$ of $\T^d$ on $\A$,
\item for all $a\in\sa{\A}$, we have
\begin{equation*}
\Lip(a) = \sup\left\{\frac{\|a-\alpha^g(a)\|_\A}{\ell(g)} : g\in H \setminus\{(1,\ldots,1)\} \right\}\text{.}
\end{equation*}
\end{enumerate}
Then if $(\A,\Lip)$ is in the closure of $\mathcal{C}$, then $\A = C^\ast(\Gamma,b)$ where $\Gamma$ is a subgroup of $\Z^d$ and $b$ one of its multiplier.
\end{corollary}

\begin{proof}
Let $(\A,\Lip)$ be in the closure of $\mathcal{C}$. Thus $(\A,\Lip)$ is the limit of a sequence $(\A_n,\Lip_n)_{n\in\N} \in \mathcal{C}$ and a sequence $(H_n)_{n\in\N}$ of compact subgroups of $\T^d$ ergodically acting on $\A_n$ via full quantum isometries. As the space of closed subsets of $\T^d$ is compact for the Hausdorff distance, up to extracting some subsequences, $(H_n)_{n\in\N}$ will converge to some closed subgroup $G$ of $\T^d$.

By Theorem (\ref{group-thm}), there exists an ergodic action of a the group $G$ on $\A$ by full quantum isometry. Our theorem  then follows from \cite[Theorem 3.3]{Albeverio80}.
\end{proof}

An interesting direction for future research exploits Theorem (\ref{group-thm}) as follows. Let $G$ be a compact Lie group. In general, it is a delicate problem to determine all C*-algebras on which $G$  acts ergodically and classify them. However, we now propose an approach to this question. Indeed, say $G$ (or some closed subgroups of $G$) acts ergodically on many matrix algebras. Theorem (\ref{group-thm}) suggests that limits of such finite dimensional C*-algebras, endowed with the natural L-seminorm from the ergodic action of $G$ (or closed subgroups) and some fixed choice of length function on $G$ \cite{Rieffel98a}, will carry ergodic actions of $G$ (when working with closed subgroups, some condition will be used to ensure the limit group is indeed $G$). We make two observations. First, we conjecture that in fact, all these C*-algebras will lie in the closure of the class of these finite dimensional C*-algebras, for the propinquity. Second, we believe that as finite dimensional actions are of course tightly related to the representations of $G$ (or projective representations, more generally), the question of finding all C*-algebras on which $G$ acts ergodically will amount to understanding under what conditions these actions give rise to converging sequences for the propinquity.

A natural group to start exploring the application of our suggested technique is $SU(2)$. Rieffel proved in \cite{Rieffel15} that all C*-algebras of continuous functions on co-adjoint orbits of $SU(2)$ are obtained as limits, for the propinquity, of full matrix algebras carrying the adjoint action of $SU(2)$ induced by well-chosen irreducible representations of $SU(2)$. More generally, it would be natural to ask whether one can obtain all possible C*-algebras of continuous sections of the sort of matrix-bundles over $SU(2)$-homogeneous spaces, which are all the examples of C*-algebras with an ergodic action of $SU(2)$ by \cite{Wassermann88}. As a prelude to this effort, we offer the following result. Note that in general, the closure of finite {\gQqcms s} contain all nuclear, quasi-diagonal {\gQqcms}, and in particular,  it contains all the quantum tori, including the non-type I variety \cite{Latremoliere13c}. It is very unclear what conditions in general are needed on a class of quantum metrics spaces built over type I C*-algebras to ensure that the closure of this class for the propinquity consists only of type I C*-algebras. Bringing together our present work with the powerful results in \cite{Wassermann88}, we however obtain:
\begin{corollary}
Let $\ell$ be a continuous length function over $SU(2)$. If $\mathcal{C}$ is the class of all finite dimensional {\gQqcms s} $(\A,\Lip)$ such that:
\begin{enumerate}
\item there exists a strongly ergodic action $\alpha$ of $SU(2)$ on $\A$ by *-automorphisms,
\item for all $a\in\sa{\A}$
\begin{equation*}
\Lip(a) = \sup\left\{ \frac{\|a-\alpha^g(a)\|_\A}{\ell(g)} : g \in SU(2), g\not= I \right\} 
\end{equation*}
\end{enumerate}
then the closure of $\mathcal{C}$ consists of {\gQqcms s} of the form $(\A,\Lip)$ where $\A$ is type I.
\end{corollary}

\begin{proof}
This results from Corollary (\ref{group-cor}) and \cite{Wassermann88}, whose corollary is that a C*-algebra admits an ergodic action of $SU(2)$ only if it is type I.
\end{proof}

\bigskip

\begin{remark}
  A referee for this paper brought to our attention \cite{Olesen80}, where a topology is introduced on group actions on Von Neumann algebras. It is an interesting project to relate their topology to the topology we define in \cite{Latremoliere18b}. A few immediate differences are as follows. Of course, our work takes place within the category of {\qcms s}, involving not only C*-algebras but also quantum metrics. In \cite{Olesen80}, the basic category is the category of Von Neumann algebras. The topology in \cite{Olesen80} is built from a topology on Von Neumann algebras acting on a fixed Hilbert space, and it is not T1. In contrast, our work is built on a metric topology on {\qcms s}. Our topology on actions in \cite{Latremoliere18b} is also metric (up to conjugacy), and accounts for the metric properties of the actions, rather than their topological properties.
\end{remark}

Another application of Theorem (\ref{main-thm}) involves inductive limits of groups. Notably, the following result does not require that the inductive limit be locally compact. This is a step in a possibly new approach to construct actions of non-locally compact groups on C*-algebras.

An inductive sequence of groups, for our purpose, is given by a sequence of groups $(G_n)_{n\in\N}$ together with group monomorphisms $\theta_n:G_n \rightarrow G_{n+1}$. We will assume all the groups $G_n$ to be topological groups with a separable, first countable topology --- hence, sequences will prove sufficient to test continuity.

The inductive limit group $G = \varinjlim_{n\rightarrow\infty}(G_n,\theta_n)$ is the quotient of $\prod_{n\in\N} G_n$ by the equivalence relation
\begin{equation*}
(g_n)_{n\in\N} \equiv (g'_n)_{n\in\N} \iff \exists N \in \N \quad \forall n \geq N \quad \theta_n(g_n) = \theta_n(g'_n) \text{.}
\end{equation*}

Let $N\in\N$. We define, for $g \in G_N$, the element $\eta_N(g)$ as the equivalence class in $G$ of
\begin{equation*}
\left(\begin{cases} \text{identity of }G_n \text{ if $n < N$,} \\ g \text{ if $n = N$,}\\
\theta_{n}(g) \text{ if $n > N$.}
\end{cases}\right)_{n\in\N}
\end{equation*}

Thus defined, the maps $\eta_n$ are group monomorphisms for all $n\in\N$. We henceforth identify $G_n$ with $\eta_n(G_n)$ in $G$ for all $n\in\N$. Thus $G = \bigcup_{n\in\N} G_n$ with $G_n\subseteq G_{n+1}$ for all $n\in\N$. 

The inductive topology on $G$ is by definition, the final topology for the maps $\theta_n$. With the above identification, it becomes the largest topology which makes, for all $n\in\N$, the inclusion of $G_n$ in $G$ continuous. In particular, $U\subseteq G$ is open in the inductive topology if and only if for all $N\in\N$, the set $U\cap G_N$ is open in $G_N$.

In our next theorem, we work directly with inductive limits seen as increasing unions of groups with the inductive space topology. As a side note, this topology does not turn $G$ to a topological group in general.

\begin{theorem}
Let $F$ be an admissible function, $\mathcal{C}$ a nonempty class of {\Qqcms{F}s} and $\mathcal{T}$ an appropriate class of tunnels for $\mathcal{C}$. 

Let $G$ be a group with unit $e$ and $(G_n)_{n\in\N}$ an increasing sequence, for inclusion, of subgroups of $G$ such that $G = \bigcup_{n\in\N} G_n$. We assume that for all $n\in\N$, the group $G_n$ is a separable first-countable topological group and that the inclusion map $G_n\hookrightarrow G_{n+1}$ is continuous. We endow $G$ with the inductive limit topology.

Let $(\A_n,\Lip_n)_{n\in\N}$ be a sequence in $\mathcal{C}$, and for each $n\in\N$, let $\alpha_n$ be an action of $G_n$ by *-automorphisms of $\A_n$.

Let $D : G \rightarrow [0,\infty)$ be a continuous function such that for all $n\in\N$ and for all $g\in G_n$
\begin{equation*}
\Lip_n\circ\alpha_n^g \leq D(g) \Lip_n \text{.}
\end{equation*}

Let $K : G\times G \rightarrow [0,\infty)$ be continuous, such that $K(g,g) = 0$ for all $g\in G$ and
\begin{equation*}
\forall n \in \N, g, h \in G_n, a \in \dom{\Lip_n} \quad \left\|\alpha_n^h(a) - \alpha_n^{g}(a)\right\|_{\A_n} \leq \Lip_n(a) K(h,g) \text{.}
\end{equation*}

If $(\A,\Lip) \in \mathcal{C}$ satisfies
\begin{equation*}
\lim_{n\rightarrow\infty} \dpropinquity{\mathcal{T}}\left((\A_n,\Lip_n),(\A,\Lip)\right) = 0\text{,}
\end{equation*}
then there exists an strongly continuous action $\alpha$ of $G$ on $\A$ such that:
\begin{enumerate}
\item for all $g\in G$, we have $\Lip\circ\alpha^g \leq D(g) \Lip$,
\item for all $g,h \in G$, and $a\in\dom{\Lip}$, we have
\begin{equation*}
\left\|\alpha^g(a) - \alpha^h(a)\right\|_\A \leq \Lip(a)  K(\delta_G(g,h)) \text{,}
\end{equation*}
\item the following inequalities hold for all $g \in G$, and for all $N\in\N$ such that $g \in G_N$
\begin{equation*}
  \liminf_{\substack{n\rightarrow\infty\\n \geq N}} \KantorovichLength{\Lip_n}(\alpha_n^{g} ) \leq \Kantorovich{\Lip}(\alpha^g) \leq \limsup_{\substack{n\rightarrow\infty\\n\geq N}} \KantorovichLength{\Lip_n}(\alpha_n^{g}) \text{,}
\end{equation*}
\end{enumerate}
\end{theorem}

\begin{proof}
For any $n\in\N$, the group $G_n$ is assumed separable, so it admits a dense countable subgroup $H_n$. If $U$ is nonempty open in $G$ then for some $N\in\N$, we have $U\cap G_N$ is open and not empty, and thus $H_N \cap U$ is not empty as well. Consequently, $H = \bigcup_{n\in\N} H_n$ is a dense subgroup $\bigcup_{n\in\N} G_n$. Notably, $H$ is countable.

Let $g \in H$. For all $n\in\N$, if $g \in H_n$ then write $\varphi_n^g = \alpha_n^g$; otherwise set $\varphi_n^g$ to be the identity of $\A_n$.

If $g,g' \in H$, there exists $N\in\N$ such that $g,g'\in G_N$. For $n\geq N$ we then have
\begin{equation*} 
\left\|\varphi_n^g\circ\varphi_n^{g'} - \varphi^{gg'}_n\right\|_{\A_n} = \left\| \alpha_n^g \circ \alpha_n^{g'} - \alpha_n^{gg'} \right\| = 0
\end{equation*}
and thus, we meet Theorem (\ref{main-thm}) hypothesis (noting that $\KantorovichLength{\Lip_n}(\alpha_n^e) = 0$ for all $n\in\N$ and $\alpha_n$ are actions by *-automorphisms). We thus conclude that there exists an action $\alpha$ of $H$ on $\A$ by *-automorphisms such that in particular
\begin{equation*}
\left\|\alpha^h(a) - \alpha^g(a)\right\|_\A \leq \Lip(a) K(h,g)\text{,}
\end{equation*}
reasoning in the same manner as in the proof of Theorem (\ref{semigroup-thm}).

Let now $g \in G$, and $N\in\N$ such that $g \in G_N$. Since $H_N$ is dense in $G_N$, there exists a sequence $(g_k)_{k\in\N}$ in $H_N$ such that $\lim_{k\rightarrow\infty} g_k = g$. Thus in particular, if $a\in\dom{\Lip}$ then
\begin{equation*}
\left\|\alpha^{g_p}(a) - \alpha^{g_q}(a)\right\|_\A \leq \Lip(a) K(g_p,g_q) \xrightarrow{p,q \rightarrow\infty} \Lip(a) K(g,g) = 0
\end{equation*}
and thus $\left(\alpha^{g_k}(a)\right)_{k\in\N}$ is Cauchy in $\sa{\A}$, which is complete. It thus converges.

Note that moreover, if we had chosen $N'\in\N$ for which $g \in G_{N'}$ and some other sequence $(h_k)_{k\in\N}$ converging in $G_{N'}$ to $g$, then, without loss of generality, we may assume $N' \geq N$; the sequence $(z_k)_{k\in\N}$ obtained by interleaving $(g_k)_{k\in\N}$ and $(h_k)_{k\in\N}$ converges to $g$ in $G_{N'}$ and thus as above, $(\alpha^{z_k}(a))_{k\in\N}$ must converge for all $a\in\dom{\Lip}$; thus both subsequences $(\alpha^{g_k}(a))_{k\in\N}$ and $(\alpha^{h_k}(a))_{k\in\N}$ have the same limit. We denote this limit by $\alpha^g(a)$.

A simple argument shows that $\alpha^g$ is a *-homomorphism and $\alpha^g\circ\alpha^{g'}(a) = \alpha^{gg'}(a)$ for all $a\in\A$, so $\alpha$ is an algebraic action of $G$ on $\A$ by *-automorphisms.

By construction, we then note that for all $a\in\dom{\Lip}$, the following inequality holds
\begin{multline*}
\left\|\alpha^g(a) - \alpha^h(a)\right\|_\A = \lim_{k\rightarrow\infty} \|\alpha^{g_k}(a) - \alpha^{h_k}(a)\|_\A \\ \leq \limsup_{k\rightarrow\infty} \Lip(a) K(g_k,h_k) =  \Lip(a) K(g,h) \text{.}
\end{multline*}

Let $a\in\sa{\A}$. Let $\varepsilon > 0$. There exists $a' \in \dom{\Lip}$ such that $\|a-a'\|_\A < \frac{\varepsilon}{3}$.

Since $K$ is continuous, there exists an open subset $U$ of $G$ such that $e \in U$ and for all $g \in U$ we have $K(g,e) < \frac{\varepsilon}{3(\diam{\A}{\Lip}\Lip(a') + 1)}$. We then have for all $g \in U$ 
\begin{align*}
\left\| \alpha^g(a) - a \right\|_\A &= \left\|\alpha^g(a) - \alpha^e(a)\right\|_\A \\
&\leq \|\alpha^g(a-a')\|_\A + \|\alpha^g(a') - \alpha^e(a')\|_\A + \|\alpha^e(a-a')\|_\A \\
&\leq 2 \|a-a'\|_\A + \frac{\varepsilon}{3}  \leq \varepsilon \text{.}
\end{align*}
We conclude that we have proven that for all $a\in\sa{\A}$
\begin{equation*}
\lim_{g\rightarrow e} \left\| \alpha^g(a) - a \right\|_\A = 0
\end{equation*}
and thus it follows immediately that $\alpha$ is strongly continuous.

Now, the same general techniques used in the proof of Theorem (\ref{group-thm}) applies to complete the theorem --- with $K$ now a function over $G\times G$.
\end{proof}

\providecommand{\bysame}{\leavevmode\hbox to3em{\hrulefill}\thinspace}
\providecommand{\MR}{\relax\ifhmode\unskip\space\fi MR }
\providecommand{\MRhref}[2]{%
  \href{http://www.ams.org/mathscinet-getitem?mr=#1}{#2}
}
\providecommand{\href}[2]{#2}

\vfill

\end{document}